\documentclass{amsart}

\author{Jack Edward Tisdell}
\title{Solving infinitary Rubik's cubes}

\usepackage{amssymb}
\usepackage{amsthm}
\usepackage{mathtools}
\usepackage{enumitem}
\usepackage{mathrsfs}
\usepackage{epigraph}
\usepackage{xcolor}
\usepackage{tikz}
\usepackage{pgfplots}
\pgfplotsset{compat=newest}

\newcommand\Z{\mathbb Z}
\newcommand\Q{\mathcal Q}

\newcommand\D{\mathcal D}
\newcommand\U{\mathcal U}
\newcommand\m{\mathbf m}
\newcommand\uc{\mathbf{uc}}
\renewcommand\S{\mathfrak S}
\newcommand\Le{\mathfrak L}
\newcommand\LLL[1]{#1\times#1\times#1}
\newcommand\red{\mathtt r}
\newcommand\white{\mathtt w}
\newcommand\green{\mathtt g}
\newcommand\orange{\mathtt o}
\newcommand\yellow{\mathtt y}
\newcommand\blue{\mathtt b}
\newcommand\NaC{\varnothing}
\newcommand\solved{\text{\upshape solved}}

\newcommand\UC{\mathrm{UC}}
\newcommand\concat{{}^\frown}
\newcommand\rest\upharpoonright

\DeclarePairedDelimiter\abs\lvert\rvert
\DeclarePairedDelimiter\seq\langle\rangle

\DeclarePairedDelimiter\card\lvert\rvert
\DeclareMathOperator\lcm{lcm}

\DeclareMathOperator\id{id}

\newtheorem{theorem}{Theorem}
\newtheorem{corollary}[theorem]{Corollary}
\newtheorem{lemma}[theorem]{Lemma}
\newtheorem{observation}[theorem]{Observation}

\theoremstyle{remark}
\newtheorem*{remark}{Remark}

\begin{document}

\begin{abstract}
    We develop infinitary analogues of the $N\times N\times N$ Rubik's cube. We'll be pushed to consider the possibility of transfinitely many twists and the foremost question we shall study is whether or not all infinite scrambles are solvable, in principle, and in how many twists. As is typical of infinitary generalizations of everyday games and puzzles, several alternative definitions are reasonable, including in particular the \emph{edged} and \emph{edgeless} cubes, which bear surprising theoretical differences, not analogous to the finite case. We show that for the edged cube of cardinality $\aleph_\alpha$, all convergent (in a suitable sense) scrambles are in fact solvable in principle in fewer than $\omega_{\alpha+1}$ many moves. For the \emph{countable} edgeless variation, we prove by entirely different methods that all convergent scrambles are solvable in a mere $\omega^2$ many moves and this solution does not require knowledge of how the scrambled configuration was obtained. Finally, we explore the space of all legal configurations of the countable edgeless cube connected to the solved configuration by accessibility. We invite several open questions, including the solvability in principle of edgeless cubes of uncountable cardinality.
\end{abstract}

\maketitle

\section{Introduction}
\epigraph{Many people not normally interested in puzzles will recall some period of their lives when they have struggled with this opponent for days at a time.}{Winning Ways for Your Mathematical Plays, Volume 4 \cite{Berlekamp2004}}

The Rubik's cube is a puzzle unlike any other. It is unrivalled in its pop cultural impact, commercial success, and enduring popularity for 50 years, especially among puzzles so difficult. Ern\H{o} Rubik, a Hungarian architecture professor, invented the puzzle in 1974 and it was first released internationally in 1980. Since then, roughly 400 million cubes have been sold worldwide. Heralded as an icon of elegant design, it now belongs to the permanent collection of the Museum of Modern Art in New York. In popular culture, it is the quintessential difficult puzzle. One of my favorite Rubik's cube cameos is its appearance in the opening sequence of the 2008 Pixar film \textit{WALL-E}. (How do you communicate to the audience that your non-verbal, non-humanoid protagonist, alone on the abandoned Earth, is profoundly curious and foreshadow his problem solving aptitude? Have him add a Rubik's cube to his collection of mysterious scavenged artifacts of course!) 

Record ``speed cubers'' can solve the puzzle in seconds. But the meta-puzzles that have enchanted mathematicians and computer scientists have taken much longer. ``God's number'', the fewest moves necessary to solve any possible scramble, was proved to be exactly 20 only as recently as 2010 and the proof required an estimated 35 CPU-years of computing time. Dozens of variations of the Rubik's cube, collectively known as ``twisty puzzles'', are available. Thanks to the MagicTile and MagicCube4D software packages, you can try your hand at twisty puzzles on Klein bottle and the four-dimensional analogue of the Rubik's cube, if you dare. The most well-known variants, though, are the $\LLL4$ ``Rubik's Revenge'' and the $\LLL5$ ``Professor's Cube''. While fun puzzles, these higher-order cubes are not mathematically much more interesting than the classic $\LLL3$ and, consequently, not much mathematical work has been undertaken concerning these variations. One notable exception is the paper of Demaine~et~al. \cite{Demaine2011} wherein they show that God's number for the $\LLL N$ cube grows asymptotically as $\Theta(N^2/\log N)$.

We are of course interested in \emph{infinite} analogues of the $\LLL N$ Rubik's cube. In the spirit of fancifully designing and studying infinitary versions of familiar games and puzzles such as infinite chess, infinite draughts, infinite nim, infinite sudoku, infinite hex, etc., we develop infinitary analogues of the $N\times N\times N$ Rubik's cube. We'll be pushed to consider the possibility of transfinitely many twists and the foremost question we shall study is whether or not all infinite scrambles are solvable, in principle, and in how many twists. As is typical of infinitary generalizations of everyday games and puzzles, several alternative definitions are reasonable, and we will explore the theoretical differences arising from the alternative possibilities. 

The infinitary Rubik's cubes we introduce are infinite games only insofar as they generalize a familiar everyday puzzle, but the theory of infinite games has no bearing on our analysis. The infinitary Rubik's cube is much closer to infinite Wordle or Mastermind (see \cite{Hamkins2022}) than to, say, infinite chess. Nonetheless, we hope that fans of the genre will find plenty to enjoy here. 

The paper is structured as follows. 
\begin{itemize}
    \item In section \ref{sec:definitions}, we give the definitions. There are a lot, this could not be helped. We define the cube itself and its color configurations and more general labellings, the twists or moves of the cube, the algebraic structure formed by these twists, and its action on the configurations/labellings. This is a lot but we think the reader will find that these all extend the corresponding ordinary Rubik's cube notions very straightforwardly. After that, we venture into less familiar territory, explaining how transfinite sequences of twists act on configurations and define the extended algebraic structure as well as the various key convergence notions. 
        
    \item In section \ref{sec:observations}, we make several important observations and technical lemmas which apply to all the variations and all cardinalities of cube. These are perhaps a little dry but they are crucial to understanding the infinitary Rubik's cube.
    \item In section \ref{sec:edged_solvable}, we give the first main result, that every convergent scramble of the edged cube of any infinite cardinality is solvable in principle. This section is short, as most of the work is done in the preceding section. We argue by showing that the underlying algebraic structure is in fact a group. This method does establish solvability (which is our primary goal), but it does not give anything resembling an algorithm for solving a given scrambled configuration and only yields the weakest possible upper bound on the longest necessary solution, namely, $<\omega_{\alpha+1}$ many moves for the edged cube of cardinality $\aleph_\alpha$. 
    \item In section \ref{sec:edgeless_algorithm}, we consider the edgeless cube variation, in particular of countable cardinality $\aleph_0$. The main result is the development of an explicit algorithm (adapting that of \cite{Demaine2011} for $\LLL N$ cubes) which solves any given accessible configuration in at most $\omega^2$ many moves, a huge improvement over the (countable) edged case in many regards. Pushing this algorithm a little further, we are able to map out the space of configurations connected to the solved configuration by one- or two-way accessibility. 
    \item Finally, in section \ref{sec:questions}, we enumerate several of the pressing open questions. 
    \item As sort of addendum, the bonus section \ref{sec:coding_orders} includes a curious result about coding well-orders into cube configurations.
\end{itemize}

\section{Preliminary definitions}
\label{sec:definitions}
The motivating---if hopelessly vague---infinitary generalizations of the ordinary $\LLL N$ Rubik's cube are the countable ``$\LLL\Z$'' and ``$\LLL{(\mathbb Q\cap[-1,1])}$'' and the continuum-sized ``$\LLL{\mathbb R}$'' and ``$\LLL{[-1,1]}$'' cubes, in which every point in each direction corresponds to a layer one might imagine twisting. Thinking about these vague suggestions reveals almost immediately which key features of ordinary Rubik's cubes will need to be incorporated in a precise way into our definitions in order to make sense of them. 

The first such feature of finite Rubik's cubes is their \emph{parity}. Namely, on $\LLL N$ cubes for \emph{odd} $N$, there is a center layer in each direction. For \emph{even} $N$, there is a central cut between two layers in each direction. The position of the center (whether it be a layer or a cut) is crucial in understanding Rubik's cubes abstractly. We will call an infinite cube odd or even depending on whether it has a distinguished center layer or a distinguished center cut. The analysis of even cubes turns out to be essentially the same as that of odd cubes, ignoring any special mention of the center layers, so the even/odd distinction is not so important. (But the existence of a distinguished center certainly is!)

Anyone who has toyed with a Rubik's cube is familiar with the special role played by the edge and corner ``cubies'', those components which have tiles from more than one face of the cube. Intuitively, whatever the ``$\LLL\Z$'' and ``$\LLL{[-1,1]}$'' cubes are supposed to be, it seems that the later \emph{should} have edge and corner cubies while the former \emph{should not}. The possibility of an Rubik's-like cube with no edge or corner cubies might seem, to a purist, like too great a departure from the traditional Rubik's cube. In our opinion, the appeal of the ``$\LLL\Z$'' cube as a natural infinitary analogue (that happens to be edgeless) is simply too great to discount. Moreover, we feel the essence of the puzzle is the manner in which twist moves act on the configuration space, rather than in particular features of the structure of the traditional cube. A stronger objection to the idea of an edgeless cube is that the analogous finite situation is uninteresting. Namely, one can view an edgeless $\LLL N$ cube simply as a ordinary $\LLL{(N+2)}$ cube with the edge and corner stickers removed, as it were. So all algorithms which apply to the ordinary cube can be applied to the finite edgeless cube. As we shall see, this turns out not to be true in the infinite case. The existence of edge and corner cubies puts severe restrictions on what can happen and, surprisingly, the analysis of the two cases is quite different.

We define first the odd, edgeless case. Given a set $L$ (for us, always assumed to be infinite), let $-L$ be a copy $L$. If $r$ is an element of $L$, we'll write $-r$ for its copy in $-L$ and we regard $-r$ and $r$ as distinct objects (so that $L$ and $-L$ are disjoint but isomorphic). Let $0$ denote a distinguished object not an element of $L$. As a notational convenience, we'll write $-(-r) = r$ for all $r \in L$ and $-0 = 0$ so that $-$ is an isomorphism between $\{0\}\cup L$ and $-L\cup\{0\}$. Set $L^\dagger = -L\cup \{0\} \cup L$. One should think of the structure $\Q_L$ we are now defining as the ``$\LLL{L^\dagger}$ Rubik's cube'', a edgeless odd cube with center layer $0$ in each direction.

For convenience, we also define the partial order $<$ on $L^\dagger$ by
\[
    -r < 0 < r
\]
for all $r \in L$. Elements of $L$ are incomparable as are elements of $-L$, as we'll see, it simply doesn't matter whether these have a non-trivial order, despite geometric intuitions.\footnote{Thanks to Joel David Hamkins for pointing this out in a spirited MathOverflow discussion.} Later, in order to describe particular configuration or algorithm, we may take the liberty to endow $L$ with an order but this is purely a matter of convenience and not part of the intrinsic structure.

Introduce two new elements $\pm\infty$, thought of as bookending $L^\dagger$ in the sense that $-\infty < -r < 0 < r < +\infty$ and write $\bar L^\dagger = [-\infty,+\infty] = L^\dagger \cup \{\pm\infty\}$. The reflection $-$ extends naturally to $\pm\infty$, namely, $-(+\infty) = -\infty$ and $-(-\infty) = +\infty$. We define the cube $\Q_L$ as (the disjoint union of) six copies of $L^\dagger \times L^\dagger$ (its faces) embedded in the $\pm\infty$ planes of $\U = \LLL{\bar L^\dagger}$ in the obvious way. Thinking of $\U$ as the ambient space in which $\Q_L$ lives, we can define $x,y,z$ coordinates (taking values in $\bar L^\dagger$) on $\U$ and we can refer to, say, corresponding points in opposite faces and planes in which we might make twists. Moreover, in $\U$, we can define the \underline Right, \underline Up, and \underline Front faces of $\Q_L$ as the faces in $+\infty$ planes in the $x$, $y$, and $z$ directions, respectively. Naturally, the \underline Left, \underline Down, and \underline Back faces are the respective opposite faces. We refer to locations $(\alpha,\beta,\gamma)$ in $\Q_L$ as \emph{cells}. That is, a cell is a point $(\alpha,\beta,\gamma) \in \U$ where exactly one of $\alpha,\beta,\gamma$ is $\pm \infty$. We emphasize as a point of clarity and simplification that the definitions do not require $L$ to have any structure beyond its set structure, so long as $L$, $-L$, and $0$ are distinguished.

A \emph{quarter-turn twist} $T_{i,\alpha} : \Q_L \to \Q_L$ for $i \in \{x,y,z\}$ and $\alpha \in \bar L^\dagger$ is a permutation of cells which rotates all cells in the $i = \alpha$ plane by a quarter turn in a designated direction and fixes all other cells. We imagine $+x,+y,+z$ as the rightward, upward, and frontward (that is, facing us, as this page) directions, respectively, and define twists according to the right-hand rule. So, for example, $T_{x,\alpha}(\alpha,\beta,+\infty) = (\alpha,-\infty,\beta)$ and $T_{x,\alpha}(\alpha,-\infty,\gamma) = (\alpha,-\gamma,-\infty)$. We say that twists $T_{i,\alpha}$ and $T_{j,\beta}$ are parallel just in case $i=j$. An important feature of finite $\LLL N$ Rubik's cubes is that a twist of the outermost layer non-trivially permutes not only the cells in the concerned layer but also all those in the adjacent face. We definitely want to allow such face twists and these are precisely the $T_{i,\pm\infty}$. For instance, $T_{x,+\infty}$ rotates the right face $\{x=+\infty\}$ through a quarter turn about $(y,z) = (0,0)$, i.e., $T_{x,+\infty}(+\infty,\alpha,\beta) = (+\infty,-\beta,\alpha)$. \textit{Nota bene}, since there are no edge or corner cubies, the face twists (non-trivially) permute \emph{only} the cells in the concerned face, quite different from face twists of ordinary finite Rubik's cubes, which necessarily involve edge and corner cubies. In fact, this is really what we mean when we say that $\Q_L$ has no edge or corner cubies.

There are some noteworthy ways in which our definitions diverge from traditional Rubik's cube notions. We mention these here in the interest of avoiding confusion among readers deeply familiar with traditional $\LLL N$ cubing. For one, we take twists as those operations which involve exactly one layer, whereas cubers often consider as single moves ``twists'' which involve all the layers between a given one and the nearest parallel face, at least in the standard notation. (We use the quotation marks only to distance this use of the word ``twist'' from our definition above, not pejoratively.) Aside from the notational conventions, this matters when one, say, counts the fewest number of moves necessary in a solve. Another possibly confusing diversion from traditional cubing is that we've defined the positive quarter-turn twists with respect to the orientation of the ambient space $\U$, whereas traditional cubers usually think of the positive quarter-turns as in the clockwise direction with respect to the nearest parallel face, so that, e.g., our Right face twist $T_{x,+\infty}$ is what the traditional $\LLL3$ cuber would think of as the \emph{reverse} Right face twist $\mathrm R'$. Our definition is, in this regard, a matter of mathematical convenience and we hope it will not cause consternation among cubers. (Although if one so desires, our definition of $\Q_L$ is perfectly well suited for the traditional quarter-turn conventions, since $L$ and $-L$ are distinguished from one another.)

A \emph{configuration} $f$ of the cube $\Q_L$ as an assignment to each cell a color in the gamut $\Gamma = \{\red,\white,\green,\orange,\yellow,\blue,\NaC\}$ (red, white, green, orange, yellow, blue, ``not a color''), that is, a map $f : \Q_L \to \Gamma$. We'll say that a configuration is \emph{legal} just in case it does not take the value $\NaC$. The value $\NaC$ is something like a $\mathtt{NaN}$ value in computing and its role will be made clear momentarily. 

The distinguished \emph{solved configuration} is the assignment
\[
    f_\solved(x,y,z) =
    \begin{cases*}
        \red & if $x = +\infty$,\\
        \blue & if $y = +\infty$,\\
        \white & if $z = +\infty$,\\
        \orange & if $x = -\infty$,\\
        \green & if $y = -\infty$,\\
        \yellow & if $z = -\infty$,\\
    \end{cases*}
\]
which assigns a distinct color to each face (in accordance with the standard Rubik's cube).

The finite sequences of quarter-turn twists form a monoid $\m_L^{<\omega}$ under concatenation generated by the individual twists. The identity element is the empty sequence $\varepsilon$ and all the parallel twists commute in $\m_L^{<\omega}$. This monoid acts on the configurations of $\Q_L$ in the obvious way, namely, given $\sigma \in \m_L^{<\omega}$ (some finite sequence of twists) and a configuration $f$, the configuration $\sigma f$ is the map $\sigma f(c) = f(\sigma^{-1}c)$ for every cell $c$. The \emph{same action} relation $\sigma \sim \tau$ whenever $\sigma f = \tau f$ for every configuration $f$ is an equivalence relation on $\m_L^{<\omega}$. The quotient $G_L = \m_L^{<\omega}/{\sim}$ is a group analogous to the Rubik's cube group. It's obvious from the definition of $\sim$ that $[\sigma]f = \sigma f$ is a well-defined action of $G_L$ on the configuration space of $\Q_L$ where $[\sigma]$ is the $\sim$ equivalence class of $\sigma$. It similarly easy to see that concatenation is a congruence with respect to $\sim$, and so induces an associative binary operation on $G_L$. To hammer home the point, the reason this quotient is a group is the following. Given any $\sigma \in \m_L^{<\omega}$, we may write $\sigma = T_k\cdots T_1$ as a composition of quarter-turn twists. Obviously, $T^3T \sim TT^3 \sim \varepsilon$ for every quarter-turn twist $T$, and inductively $(T_1^3\cdots T_k^3)(T_k\cdots T_1) \sim (T_k\cdots T_1)(T_1^3\cdots T_k^3)\sim \varepsilon$. Thus, $[T_1^3\cdots T_k^3] = [\sigma]^{-1}$ in $G_L$. In other words, the inverses in $G_L$ are (represented by) the inverse sequences in the Rubik's cube sense. Note the order reversal involved.

For the $\LLL N$ cube, the distinction between this monoid and this group is not usually made explicitly. Our primary interest in what follows essentially concerns the study of a certain natural extension of this monoid given by the suitably defined convergent (transfinite) sequences of twists and its action on the configuration space of $\Q_L$. We will be able to extend the relation $\sim$ as well but it is not at all clear whether or not the quotient forms a group, as we'll see, the representatives we produced above for the inverses of finite sequences do not generally have analogues for infinite sequences. 

It's clear even on cursory consideration that if $\Q_L$ is to be a novel generalization of the $\LLL N$ cube, we will want to allow infinitely many twists so long as we still obtain a well-defined configuration thereafter. Of course, if we can do this, there is no principled reason to stop---just keep applying twists to the limiting configuration---so the process is inherently transfinite. Let's make this precise. Given a $\theta$-sequence $\seq{\sigma_\eta : \eta < \theta}$ of elements of $G_L$ for some ordinal $\theta$ and an initial configuration $f_0$, we define the subsequent configurations by $f_{\eta + 1} = \sigma_\eta f_\eta$ and for limit stages $\lambda$, we set $f_\lambda(c) = \lim_{\eta \nearrow \lambda} f_\eta(c)$ if this limit exists and $f_\lambda(c) = \NaC$ otherwise. In other words, at successor stages, we act on the configuration by the next element of the sequence and at limit stages $\lambda$, a cell has color $\gamma$ if and only if its color had eventually stabilized on $\gamma$ before stage $\lambda$ and otherwise it has color $\NaC$ (which is a way of saying it has no color). Note that when $\theta$ is a limit ordinal, this definition yields the final configuration $f_\theta$ (so the sequence of configurations always has a terminal element, regardless of the order type of the sequence of twists).

There are a few different convergence notions we wish to distinguish. Consider the sequence $\seq{\sigma_\eta : \eta < \theta}$ applied to an initial configuration $f_0$. We'll say the resulting \emph{sequence of configurations $\seq{f_\eta: \eta \le \theta}$ is convergent} just in case $f_\theta$ is legal and divergent otherwise. In this case, we'll say that \emph{the sequence of twists $\seq{\sigma_\eta: \eta < \theta}$ is convergent over $f_0$}. Crucially, these convergence notions depend on the initial configuration as well as the sequence of twists. Of course, there is nothing special about functions $f : \Q_L \to \Gamma$ which take color values. Twist sequences act in exactly the same way on the space of functions $f : \Q_L \to X\cup\{\NaC\}$ for any set $X$ which does not contain $\NaC$. We refer to such arbitrary-valued functions on $\Q_L$ as \emph{labellings} (so that, in particular, configurations are $\Gamma$-valued labellings). Notice that if a sequence of twists is convergent over some labelling $f : \Q_L \to X\cup\{\NaC\}$ and $g : X \to Y$ is any function with $\NaC \not\in Y$, then the sequence is convergent over the labeling $g\circ f$. This is essentially because the composition $g\circ f$ is no more discriminating among cell values than $f$. We say that a sequence of twists is \emph{universally convergent} if it is convergent over the identity labeling $\id_{\Q_L} : \Q_L \to \Q_L\cup\{\NaC\}$. Universal convergence is a property of the twist sequence alone. The sense in which this convergence is universal is that, by the composition property, a universally convergent sequence is convergent over every labeling. The intended understanding of universally convergent sequences is as follows. The identity labeling is simply a canonical labeling that assigns a unique value to each cell. So the universally convergent sequences are those which track (through limit stages) not merely the color of each cell, but the original position of the tiles in each cell.

Notice that, \textit{prima facie}, in exploring what's possible with infinite sequences of twists, it does not suffice to consider only a generating set of $G_L$, in particular the quarter-turn twists $T_{i,\alpha}$. For example, for any quarter-turn twist $T$, the $\omega$-sequence $T,T,T,\dots$ generally does not yield a legal limiting configuration whereas the sequence $T^4,T^4,T^4,\dots$ does (trivially since $T^4$ is the identity). But if one insists on thinking of $T^4$ as four applications of $T$, then it would seem that these are the same sequence of twists, obviously an unsatisfactory situation. At the heart of the matter is the fact that this limit definition (which we argue is perfectly natural) does not respect the group action. Namely, although $(\tau\sigma)f$ and $\tau(\sigma f)$ are the same configuration, it is important to keep track of whether one has applied $\sigma$ then $\tau$ or simply applied $\tau\sigma$ in determining the limiting configuration. Thus, we must be mindful of the entire history of configurations and, moreover, we should be clear on which compositions of quarter-turn twists ought to count as single operations. 

We refer to quarter-turns $T_{i,\alpha}$, half-turns $T_{i,\alpha}^2$, and reverse quarter-turns $T_{i,\alpha}^3 = T_{i,\alpha}^{-1}$ of a single layer as \emph{basic twists}. Obviously, the basic twists generate $\m_L^{<\omega}$ since the quarter-turns alone do. It seems reasonable to allow all and only the basic twists as individual operations. Certainly we want the quarter turns to count as single moves and there seems no reason to then disallow the reverse quarter turns. The status of half turns as single operations is perhaps more debatable but is certainly in the spirit of traditional cubing, for example, in the infamous result that God's number (the maximum number of moves necessary to solve the classic $\LLL3$ Rubik's cube) is 20, half turns count as single moves. By \emph{basic sequences}, we mean sequences of basic twists of arbitrary ordinal length. The basic sequences form a monoid $\m_L$ under concatenation extending $\m_L^{<\omega}$ and acting on the configurations of $\Q_L$ as described above. 

Recall, the motivation for the definition of the equivalence $\sim$ on $\m_L^{<\omega}$ is that sequences should be equivalent just in case they yield the same final configuration for any given initial configuration, and this definition works just the same for $\m_L$. Namely, if $\vec\sigma = \seq{\sigma_\eta : \eta < \theta}$ and $\vec\tau = \seq{\tau_\eta : \eta < \zeta}$, then we'll say $\vec \sigma \sim \vec \tau$ just in case $\vec \sigma f = \vec \tau f$ for every configuration $f$ (where by $\vec \sigma f$, we mean the terminal configuration in the sequence obtained by applying $\vec \sigma$ to the initial configuration $f$). Like the finite case, it follows immediately from the definition of $\sim$ that that that concatenation of (transfinite) sequences is a congruence with respect to $\sim$, and so quotient $\m_L/{\sim}$ is a monoid. But now, the representatives for inverses have no analogues, since the reverse of a infinite well-order is not a well-order.

When considering the action of explicit sequences, we will compactly summarize the situation using the pipe symbol $|$ in the following way. $\dots \sigma_2|\sigma_1|\sigma_0|f$ for $\sigma_k \in G_L$ denotes the sequence of configurations $f, \sigma_0f, \sigma_1\sigma_0f,\dots$. We will also use the notation $\seq{\sigma_\eta:\eta < \theta}|f$ to mean the same thing, where possibly $\theta > \omega$. So for instance, $T_{i,\alpha}^2|f$ and $T_{i,\alpha}|T_{i,\alpha}|f$ are distinct sequences of configurations. The former is applies single half-turn twist of the $i=\alpha$ layer while the later applies two quarter turns of that same layer in succession. Both end in the same configuration $T_{i,\alpha}^2f$ but the later passes through the intermediate quarter-turn. Abusing notation slightly, we may write $f' = s|f$ for some sequence $s$ to mean that $f'$ is the terminal configuration of the sequence $s|f$. 

A basic sequence is said to be \emph{twist-finite} if each basic twist appears in it only finitely many times. Since there are infinitely many layers, a twist-finite sequence certainly need not be finite. A simple example is the sequence which applies one quarter-turn twist in every $x$ layer (in any order), amounting to a global rotation.

Finally, we'll say a configuration $f$ is \emph{accessible} from $f_0$ if there exists a basic sequence $\seq{\sigma_\eta : \eta < \theta}$ such that $f$ is the terminal configuration of $\seq{\sigma_\eta : \eta<\theta}|f_0$. If we don't specify $f_0$, it is understood to be the solved configuration $f_\solved$. A $\theta$-\emph{scramble} is a basic sequence of length $\theta$ starting at $f_\solved$, that is, the entire sequence of configurations. 

We define the \emph{edged cube} $\bar \Q_L$ as follows. 
Like before, $\bar\Q_L$ which we think of as the ``$\LLL{\bar L^\dagger}$ Rubik's cube'' consists of six copies of $\bar L^\dagger \times \bar L^\dagger$ embedded in the extreme planes of $\U = \LLL{\bar L^\dagger}$. The key difference is that now we have edge and corner cubies, which carry more than one color, so there is not a one-to-one correspondence between cells (to which we assign colors) and locations in $\U$. So we augment the extreme locations in $\U$ (i.e., those having at least one $\pm\infty$ component) with a indication of the face to which the cell belongs. We define the edged cube $\bar\Q_L$ as the set of these augmented triples. In the case that there is only one $\pm\infty$ component, this is unambiguous, otherwise, we'll denote a cell by a triple $(\alpha,\beta,\gamma)$ with one component underlined indicating the face to which the cell belongs. For example, $(-\infty,\beta,\underline{+\infty})$ is the front cell of a cubie in the Front Left edge and $(+\infty,\underline{+\infty},+\infty)$ is the Up cell of the Up Right Front corner cubie. Configurations, labellings, and twists are all defined as before with $\bar\Q_L$ playing the role of $\Q_L$ and now face twists $T_{i,\pm\infty}$ act as before on the cells in the concerned face but also on the shared corner and edge cubies. So $T_{x,+\infty}(\underline{+\infty},\beta,\gamma) = (\underline{+\infty}, -\gamma, \beta)$ and $T_{x,+\infty}(+\infty,\beta,\underline{+\infty}) = (+\infty,\underline{-\infty},\beta)$ and so on.

Lastly, we can also define the \emph{even} edgeless and edged cubes analogously but with $L^\ddagger = -L \cup L$ replacing $L^\dagger = -L\cup\{0\}\cup L$. Like an $\LLL N$ cube for $N$ even, the infinite even cubes have no center cells in any row or column of each face but the center cut between $-L$ and $L$ is distinguished. All the arguments we give for the odd edged and edgeless cubes apply to the even edged and edgeless cubes, respectively, simply ignoring mention of center cells, except where otherwise stated.

It will be useful to have terminology for the various types of cells. Cells having having two equal or opposite coordinates (e.g., $(\pm\alpha,\pm\alpha,\underline{+\infty})$) are called \emph{diagonal} cells. Cells (in the edged cube) having exactly two $\pm\infty$ coordinates are called \emph{edge} cells and cells having three $\pm\infty$ coordinates are \emph{corner} cells. Cells (in the odd cube) having at least one $0$ coordinate are called \emph{cross} cells and the six cells with two $0$ coordinates are the \emph{center} cells.

\section{General observations and lemmas}
\label{sec:observations}
Like an $\LLL N$ cube, the label of any given cell can be moved by basic sequences to at most 24 different cells. That is, for any given cell $c$, there are at most 24 cells with label $c$ after applying a basic sequence to the identity labeling. We will call this set of cells the \emph{cluster} of cell $c$ and denote it by $C(c)$. In fact, each cluster has exactly 24 cells except the center cluster, which contains all and the only the six center cells. Every non-center cluster $C$ has a unique representative $(x,y,\underline{+\infty})$ with $x > 0$ and $y \ge 0$ in the upper right quadrant of the Front face (including the cross cells $(x,0,\underline{+\infty})$ but excluding $(0,y,\underline{+\infty})$). We'll sometimes denote this cluster by $C(x,y)$ and the center cluster by $C(0,0)$. Each non-center cross cluster $C(x,0)$ contains the cells with face coordinates $(0,\pm x)$ and $(\pm x,0)$. Each diagonal cluster $C(\alpha,\alpha)$ contains all and only the cells with face coordinates $(\pm \alpha, \pm \alpha)$. All and only the edge cross cells (if they exist) belong to the cluster $C(+\infty,0)$. All and only the corner cells belong to the corner cluster $C(+\infty,+\infty)$. Every other edge cell belongs to a cluster distinct from that of its conjoined parter on the adjacent face. For instance, the cell $(x,+\infty,\underline{+\infty}) \in C(x,+\infty)$ for $0 < x < +\infty$ while $(x,\underline{+\infty},+\infty) \in C(+\infty,x)$. Every basic twist acts on each cluster by a permutation which may be thought of as an element of $S_{24}$, the symmetric group of order $24!$.

With this in mind, we make the following observation.

\begin{observation}
    Illegality is persistent in the sense that no legal configuration is accessible from an illegal one (in any ordinal number of moves).
\end{observation}

\begin{proof}
    Suppose $f_0$ is an illegal configuration, i.e., $f_0(c_0) = \NaC$ for some cell $c_0$. Let $\seq{\sigma_\eta : \eta < \theta}$ be a basic sequence and suppose that every configuration $f_\eta$ in the sequence $\seq{\sigma_\eta : \eta < \theta}|f_0$, except possibly $f_\theta$, has a colorless cell $c_\eta$ in the cluster of $c_0$, that is, $f_\eta(c_\eta) = \NaC$ for some $c_\eta \in C(c_0)$ for every $\eta < \theta$. If $\theta = \xi + 1$ is a successor ordinal, then $f_\theta(\sigma_\xi c_\xi) = f_{\xi+1}(\sigma_\xi c_\xi) = \sigma_\xi f_\xi(\sigma_\xi c_\xi) = f_\xi(c_\xi) = \NaC$ so $\sigma_\xi c_\xi \in C(c_0)$ is colorless in $f_\theta$. On the other hand, if $\theta$ is a limit ordinal, then since $C(c_0)$ is finite and has at least one colorless cells for each $f_\eta$, some cell $c \in C(c_0)$ must have color $\NaC$ unboundedly often, i.e., there is subsequence $\eta_\zeta$ unbounded in $\theta$ such that $f_{\eta_\zeta}(c) = \NaC$. Then, $f_\theta(c) = \NaC$ (either because it stabilized on $\NaC$ or never stabilized).
\end{proof}

This is a nice robustness result which dispels some worries one might have had about the set-up. For one, we are justified in interpreting illegal configuration as simply ``ill-defined'' with respect to coloring (or more general labelling). Moreover, for a sequence $s$ of twists, we defined the sequence $s|f_0$ of configurations as convergent just in case the terminal configuration is legal. The previous observation ensures us that such a sequence could not have passed through an intermediary illegal configuration, that is, every extension of a divergent sequence is divergent. 

Next, let's see that universally convergent sequences act on the identity labelling bijectively on every cluster.

\begin{observation}
    If $s = \seq{\sigma_\eta : \eta < \theta}$ is a universally convergent basic sequence and $f$ is the final configuration obtained by applying $s$ to the identity labelling $\id_{\Q_L}$, then $f : \Q_L \to \Q_L$ is a bijection. Moreover, $f(C) = C$ for each cell cluster $C$. The same is true of $\bar\Q_L$.
    \label{obs:uc_bijective}
\end{observation}

\begin{proof}
    We proceed by induction on $\theta$. Let $\seq{f_\eta : \eta \le \theta}$ be the configurations obtained by applying $s = \seq{\sigma_\eta : \eta < \theta}$ to $\id_{\Q_L}$. $f_0 = \id_{\Q_L}$ is clearly a bijection which fixes every cluster. Assume that $f_\eta : \Q_L \to \Q_L$ is a bijection fixing every cell cluster for each $\eta < \theta$. If $\theta = \zeta + 1$ is a successor ordinal, then $f_\theta = \sigma_\zeta f_\zeta = f_\zeta \circ \sigma_\zeta^{-1}$ is obviously a bijection. Suppose $\theta$ is a limit ordinal. Since $s$ is universally convergent, $f_\theta$ is legal. Suppose then that $f_\theta(c) = f_\theta(c') = d \in \Q_L$ for some cells $c,c',d \in \Q_L$. By definition, the sequences $f_\eta(c)$ and $f_\eta(c')$ for $\eta < \theta$ much each eventually stabilize on $d$. So for some large enough $\zeta < \theta$, we have $f_\zeta(c) = f_\zeta(c')$, hence $c=c'$ by injectivity of $f_\zeta$. Thus, $f_\theta$ is injective. Let $e \in \Q_L$ be any cell, and $C$ its cluster. By assumption, for each $\eta < \theta$, there is some cell $c_\eta \in C$ with $f_\eta(c_\eta) = e$. Then by finiteness of $C$, there is a subsequence $\eta_\nu$, unbounded in $\theta$, on which $c_{\eta_\nu} = c$ is constant. Then $f_\theta(c)$ is either $e$ or $\NaC$, but the later is impossible by legality. Thus, $f_\theta$ is surjective. The same argument works for $\bar\Q_L$.
\end{proof}

This justifies our earlier understanding of universally convergent sequences as tracking the identity of each piece of the cube through limit stages. It shows also that the effect of a universally convergent sequence can be understood as an infinitary product $\prod_C \pi_C$ of permutations $\pi_C$ of each cluster $C$. (This is coherent since the clusters are disjoint.) Each $\pi_C$ may be viewed as an element of $S_{24}$, the symmetric group of order $24!$. In fact, this leads to the following equivalence result.

\begin{lemma}
    The universally convergent sequences are exactly the twist-finite sequences
    \label{lemma:UC=TF}
\end{lemma}

\begin{proof}
    Each cell is affected by only nine basic twists, the three basic twists in each slice containing it and the three basic twists of its face. It follows immediately that every twist-finite sequence is universally convergent, since each cell is affected only finitely many times by any such sequence. 

    The converse is more subtle. Suppose $\seq{\sigma_\eta : \eta < \theta}$ is a universally convergent sequence but is not twist finite. I.e., there is some basic twist $T$ appearing infinitely many times in the sequence, say along a subsequence $T = \sigma_{\eta_\nu}$. Without loss of generality, we may assume that $\theta$ is a limit ordinal and the $\eta_\nu$ are unbounded in $\theta$, otherwise just truncate the sequence at a limit of the $\eta_\nu$ (e.g., after the first $\omega$ many). Let $c$ be any cell which is affected non-trivially by $T$. As usual, let $\seq{f_\eta : \eta \le \theta}$ be the sequence of labellings obtained by applying $\seq{\sigma_\eta: \eta<\theta}$ to the identity labelling. They are all legal since the twist sequence is universally convergent. By legality of $f_\theta$, we know $\lim_\eta f_\eta(c) = f_\theta(c)$ exists the and same is true of $f_\theta(Tc) = \lim_\eta f_\eta(Tc)$. But then $f_\theta(Tc) = \lim_\eta f_\eta(Tc) = \lim_\nu f_{\eta_\nu + 1}(Tc) = \lim_\nu \sigma_{\eta_\nu}f_{\eta_\nu}(Tc) = \lim_\nu Tf_{\eta_\nu}(Tc) = \lim_\nu f_{\eta_\nu}(c) = f_\theta(c)$. We chose $c$ so that $Tc \ne c$, thus violating the injectivity of $f_\theta$. So there can be no such twist $T$.
\end{proof}

With this in mind, we can consider the algebraic structure of the universally convergent sequences. First, the following is immediate from the previous lemma.

\begin{corollary}
    The universally convergent basic sequences form a submonoid $\uc_L$ of $\m_L$.%
    \label{cor:uc_monoid}
\end{corollary}
\begin{proof}
    The concatenation of twist-finite sequences is twist-finite.
\end{proof}

Next, we establish the following equivalent definition of $\sim$ for universally convergent basic sequences.

\begin{lemma}
    Let $\vec\sigma,\vec\tau \in \uc_L$ be universally convergent basic sequences. Then $\vec\sigma \sim\vec\tau$ if and only if $\vec\sigma\id = \vec \tau \id$. That is, they are equivalent just in case they result in the same labelling when applied to the identity labelling. Moreover, if either holds, then in fact, $\vec\sigma f = \vec\tau f$ for every labelling $f$. (Recall, the definition only refers to $\Gamma$-valued configurations, not arbitrary labellings.)
    \label{lemma:equiv_id}
\end{lemma}

\begin{proof}
    First suppose that $\vec\sigma\id = \vec\tau\id$. Let $f$ be any labelling (including any configuration). We want to show that $\vec\sigma f = \vec \tau f$. Consider any cell $c \in \Q_L$. Since $\vec\sigma$ is twist finite, $\vec \sigma$ affects $c$ non-trivially only on a finite subsequence, say $\seq{\sigma_{\eta_0},\dots,\sigma_{\eta_k}}$. Thus,
    \begin{align*}
        \vec\sigma f(c)
        &= (\seq{\sigma_{\eta_0},\dots,\sigma_{\eta_k}} f)(c)
        \\&= f( \sigma_{\eta_k}^{-1}\cdots \sigma_{\eta_0}^{-1} c)
        \\&= f( \id(\sigma_{\eta_k}^{-1}\cdots \sigma_{\eta_0}^{-1} c))
        \\&= f( (\seq{\sigma_{\eta_0},\dots,\sigma_{\eta_k}} \id)(c) )
        \\&= f( (\vec\sigma\id)(c) ).
    \end{align*}
    (More colloquially, this says that the label/color under $\vec\sigma f$ of the cell $c$ is just the label/color under $f$ of the cell $(\vec\sigma \id)(c)$, i.e., the cell whose label is at $c$ after acting on the identity labelling by $\vec\sigma$, which is exactly how we expect universally convergent sequences to behave.) Likewise, $\vec\tau f(c) = f( (\vec\tau\id)(c) )$. But by assumption, $\vec\sigma\id = \vec\tau\id$, hence $\vec\sigma f = \vec\tau f$, as desired.

    Conversely, suppose that $\vec\sigma \id \ne \vec\tau\id$ for $\vec\sigma,\vec\tau \in \uc_L$. We will find a configuration $f : \Q_L \to \Gamma$ such that $\vec\sigma f \ne \vec\tau f$. By assumption, there is some cell $c \in \Q_L$ where $(\vec\sigma\id)(c) \ne (\vec\tau\id)(c)$. By the same reasoning as above, $\vec\sigma f(c) = f( (\vec \sigma \id)(c) )$ and $\vec\tau f = f( (\vec\tau\id)(c) )$. Thus, any configuration $f$ which assigns distinct colors to $(\vec\sigma\id)(c)$ and $(\vec\tau\id)(c)$ does the trick.
\end{proof}

\begin{remark}
    The above lemma does not hold in general for arbitrary sequences in $\m_L$. More specifically, one can easily construct two (necessarily not universally convergent) sequences $\vec\sigma,\vec\tau \in \m_L$ and a configuration $f$ such that $\vec\sigma\id = \vec\tau\id$ but $\vec\sigma f \ne \vec\tau f$. 
\end{remark}

\begin{lemma}
    $\uc_L$ is closed under $\sim$. That is, if $\vec\sigma$ is a universally convergent basic sequence and $\vec\sigma \sim \vec\tau$ for some basic sequence $\vec\tau$, then $\vec\tau$ is universally convergent.
    \label{lemma:un_closed_eq}
\end{lemma}
\begin{proof}
    Fix $\vec \sigma \in \uc_L$. Suppose $\vec\tau$ is not universally convergent but $\vec\tau \sim \vec \sigma$. Say $\vec\tau = \seq{\tau_\eta : \eta < \theta}$ and let $\seq{f_\eta : \eta \le \theta}$ be the sequence of labellings obtained by acting on $f_0 = \id$ by $\vec\tau$. By definition, $f_\theta(c) = \NaC$ for some cell $c$, say in cluster $C$. Without loss of generality, we may assume that $\theta$ is a limit ordinal and that every $f_\eta$ for $\eta < \theta$ is legal (otherwise, just consider the first stage where an illegal labelling appears). Thus, there is a subsequence $\seq{\eta_\nu}$, unbounded in $\theta$, such that $f_{\eta_\nu}(c) \ne f_{\eta_\nu+1}(c)$ for every $\nu$ (i.e., the label of $c$ changes unboundedly often before stage $\theta$). Moreover, every proper initial segment of $\vec \tau$ is universally convergent, hence each $f_\eta$ for $\eta < \theta$ restricted to $C$ is a permutation, by Observation~\ref{obs:uc_bijective}. Since $C$ is finite, by passing to a further subsequence (for which we use the same notation), we may assume $f_{\eta_\nu}(c)$ is constant in $\nu$, say taking value $d$. Then passing to a sub-subsequence, we may further assume that $f_{\eta_\nu+1}(c)$ is constant in $\nu$, say taking value $d'$, necessarily different from $d$. 

    Now take any configuration $g_0$ which assigns distinct colors to $d$ and $d'$. Since $\vec\sigma$ is universally convergent, we know $\vec\sigma g_0$ is a legal configuration but we claim that $\vec\tau g_0$ is not. Let $\seq{g_\eta : \eta \le \theta}$ be the sequence of configurations obtained by applying $\vec\tau$ to $g_0$. By similar reasoning as in the previous lemma, $g_\eta = g_0\circ f_\eta$ for every $\eta < \theta$. (Here we are using that fact that the proper initial segments of $\vec\tau$ are universally convergent.) But then $g_{\eta_\nu}(c) = g_0(f_{\eta_\nu}(c)) = g_0(d)$ and $g_{\eta_\nu+1}(c) = g_0(f_{\eta_\nu+1}(c)) = g_0(d')$ are different for every $\nu$. Hence $(\vec\tau g_0)(c) = g_\theta(c) = \NaC$. 
\end{proof}

Putting together the results of this section, we reveal the algebraic structure of the universally convergent basic sequences.

\begin{lemma} 
    The quotient $\UC_L = \uc_L/{\sim}$ forms a group extending $G_L$. Furthermore, every element of this group has order no greater than the least common multiple of the orders of the elements of $S_{24}$.
    \label{lemma:UC_group}
\end{lemma}

\begin{proof}
    By Corollary~\ref{cor:uc_monoid}, we know that concatenation is an associative binary operation on $\uc_L$ (with identity the empty sequence). So see that $\UC_L$ is a monoid, we must show that concatenation induces a binary operation on the quotient. Namely, let $[\vec\tau][\vec\sigma] = [\vec\tau\vec\sigma]$ where $\vec\tau\vec\sigma$ is $\vec\sigma$ followed by $\vec\tau$. We need that this definition does not depend on the representatives $\vec\sigma,\vec\tau$, i.e., that concatenation of universally convergent basic sequences is a congruence with respect to $\sim$. Let $\vec\sigma,\vec\tau \in \uc_L$ and say $\vec\sigma' \sim \vec\sigma$ and $\vec\tau'\sim\vec\tau$. (By Lemma~\ref{lemma:un_closed_eq}, we find that $\vec\sigma',\vec\tau' \in \uc_L$, even if we do not assume so.) By Lemma~\ref{lemma:equiv_id}, we find that $\vec\sigma'\id = \vec\sigma\id$. This is itself a labelling, hence applying the lemma again yields $\vec\tau'(\vec\sigma'\id) = \vec\tau(\vec\sigma\id)$, and thus $(\vec\tau'\vec\sigma')\id = (\vec\tau\vec\sigma)\id$. Finally, applying the lemma a third time (in the converse direction, as it were) gives $\vec\tau'\vec\sigma' \sim \vec\tau\vec\sigma$, as required.

    The more interesting part is the existence of inverses in $\UC_L$. Again, it's worth remarking that the argument we gave for finite basic sequences breaks down completely because the inverse in the ordinary Rubik's cube sense (i.e., the reverse sequence with each basic twist inverted) is not even well-ordered for infinite sequences! Nonetheless, we can find an inverse as follows. By Observation~\ref{obs:uc_bijective}, every universally convergent sequence acts on labellings by permuting each cluster, that is, as a product $\prod_C \pi_C$ over all clusters $C$ of permutations $\pi_C$ which may be regarded as elements of $S_{24}$. This infinitary product is coherent since the clusters are disjoint. Also by disjointness of clusters, $\big(\prod_C \pi_C\big)^k = \prod_C \pi_C^k$ for any finite $k$. So taking $k$ to be the least common multiple of the orders of the elements of $S_{24}$, we have that $\big(\prod_C \pi_C\big)^k = \prod_C \pi_C^k = \prod_C \id_C = \id$. Therefore, the $(k-1)$-fold concatenation of the original sequence is itself twist-finite, hence universally convergent, and represents a (two-sided) inverse of the original sequence.
\end{proof}

This is quite nice but from the perspective of the Rubik's cube puzzle, it's not really what we're after. Universal convergence is too strong. We only really care about convergence over the solved configuration, \emph{prima facie}, a weaker property. Surprisingly, the situation differs starkly between the edgeless and edged cases. 

\section{The edged cube of any infinite cardinality is solvable in principle}
\label{sec:edged_solvable}
In this section, we consider only the edged cube $\bar\Q_L$. We will characterize convergence over the solved configuration and prove that every accessible configuration is solvable, in principle. The key ingredient is the following lemma which concerns the uniqueness of coupled edge and corner cells.

\begin{lemma}
    Suppose $f$ is a legal configuration accessible from $f_\solved$. Then for each pair $\gamma,\gamma' \in \Gamma$ of (non-$\NaC$) colors and each (non-corner) edge cluster $C$, there is at most one pair of coupled edge cells $c,c'$ with $c \in C$ and $f(c) = \gamma$ and $f(c') = \gamma'$. Similarly, for each triple $\gamma,\gamma',\gamma'' \in \Gamma$, there is at most one triple of mutually adjacent corner cells $c,c',c''$ having those three colors under $f$.
    \label{lemma:color_coupling}
\end{lemma}

\begin{proof}
    Let $\seq{\sigma_\eta : \eta < \theta}$ be a basic sequence convergent over $f_\solved$ and let $\seq{f_\eta : \eta \le \theta}$ be the sequence of configurations obtained by applying $\seq{\sigma_\eta: \eta < \theta}$ to $f_\solved$. We argue by induction on $\theta$. Clearly, $f_0 = f_\solved$ has the desired property. (A color pair determines an edge, if there is one, the edge cluster $C$ then determines a unique pair of coupled cells in that edge, and any color triple matches at most one corner.) Individual basic twists preserve this property, so the successor $\theta$ case is trivial. So assume $\theta$ is a limit ordinal and that all $f_\eta$ for $\eta < \theta$ have the desired property. If there is more than one pair or triple of cells witnessing a violation of the desired property under $f_\theta$, then by legality, the color of each of those two or three cells must have stabilized before stage $\theta$, and thus, there was already a violation of the desired property at some earlier stage $\eta < \theta$, contrary to our hypothesis. 
\end{proof}

With this in mind, we can refine the proof of Lemma~\ref{lemma:UC=TF} to show that, for the edged cube, our various natural convergence notions coincide.

\begin{lemma}
    For the edged cube $\bar\Q_L$, the basic sequences convergent over the solved configuration, the universally convergent sequences, and the twist-finite sequences coincide.
    \label{lemma:edged_convergence}
\end{lemma}

\begin{proof}
    Having already shown the universally convergent sequences are exactly the twist-finite sequences, it suffices to show that the twist-finite sequences are precisely those convergent over the solved configuration.

    Again, one direction is obvious: every twist-finite sequence is universally convergent (since each cell is affected only finitely many times), hence convergent over the solved configuration.

    Conversely, suppose $\seq{\sigma_\eta : \eta < \theta}$ a basic sequence convergent over $f_\solved$ and let $\seq{f_\eta : \eta \le \theta}$ be the corresponding sequence of configurations. Assume some basic twist $T$ occurs infinitely many times in $\seq{\sigma_\eta : \eta < \theta}$. Arguing as in the proof of Lemma~\ref{lemma:UC=TF}, we may assume the $\eta$ for which $\sigma_\eta = T$ are unbounded in $\theta$ and we find that $f_\theta(Tc) = f_\theta(c)$ for every cell $c$ affected non-trivially by $T$. 

    For concreteness, suppose $T = T_{x,\alpha}^k$ for some $\alpha \in \bar L^\dagger$ and $k = 1,2,3$, that is, $T$ is a twist of the $x = \alpha$ slice. First consider the case that $\alpha \ne \pm\infty$, so $T$ is a twist of one of the internal vertical slices, according to our usual conventions. Take $c$ to be any one of the eight edge cells affected by $T$, e.g., $c = (\alpha,+\infty,\underline{+\infty})$, the uppermost cell of the $\alpha$ column of the front face, and let $c'$ be its coupled edge cell on the adjacent face. Let $\gamma = f_\theta(c)$ and $\gamma' = f_\theta(c')$. Since $f_\theta$ is legal, $\gamma,\gamma' \ne \NaC$ and since $c,c'$ are coupled, $\gamma \ne \gamma'$.  We have deduced that $\gamma = f_\theta(c) = f_\theta(Tc)$ for the distinct cells $c,Tc$ in the cluster $C(c)$ and likewise for $c'$ and $\gamma'$. Thus, under $f_\theta$, there are two distinct cells in cluster $C(c)$ with color $\gamma$ whose coupled cells have color $\gamma'$. But there is at most one such cell in each (non-corner) edge cluster in any configuration obtained from $f_\solved$, by the previous lemma.

    This argument fails if $\alpha = \pm\infty$, i.e., if $T$ is a face twist, because the corner cells all belong to the same cluster and there may be two distinct pairs of coupled corner cells having the same pair of colors. But, in the case of corner cells, there is a third cell coupled to each pair and these two cells must have different colors but still be related by $T$. Let's make this precise. Let $c = (\alpha,+\infty,\underline{+\infty})$ be the one of the upper corner cells in the Front face affected by $T$. Let $c'$ be its coupled cell in the Up face and $c''$ be its coupled cell in the Left or Right face. Then $f_\theta(Tc) = f_\theta(c)$ and likewise for $c'$ and $c''$. But also $Tc$, $Tc'$, and $Tc''$ are mutually coupled corner cells distinct from $c,c',c''$ and so we've found two distinct triples of coupled corner cells with the same three colors, contrary to the previous lemma.

    The argument is analogous if $T$ is a $y$ or $z$ twist. Thus, there can be no such twist $T$ and the sequence is twist-finite.
\end{proof}

\begin{theorem}
    The basic sequences convergent over $f_\solved$ modulo $\sim$ form the group $\UC_L$ (of Lemma~\ref{lemma:UC_group}) under concatenation. In particular every convergent $\theta$-scramble is solvable in at most $\theta\cdot(k-1)$ many moves where $k = \lcm\{\abs\pi : \pi \in S_{24}\}$.
    \label{thm:edged_group}
\end{theorem}

\begin{proof}
    By the previous theorem, every convergent scramble is a obtained by twist-finite sequence $\vec\sigma$, say of length $\theta$. The sequence $\vec\sigma^{k-1}$ is clearly itself twist-finite, hence universally convergent, and by Lemmas~\ref{lemma:UC=TF} and \ref{lemma:UC_group}, the sequence $\vec\sigma^k$ leaves the cube unchanged. Thus, the sequence $\vec\sigma^{k-1}$ of length $\theta\cdot(k-1)$ solves the scramble by $\vec\sigma$.
\end{proof}

This gives us a crude upper bound on the length of the shortest solve over all possible scrambles.

\begin{corollary}
    If $\card L = \aleph_\alpha$, then God's number for the edged cube $\bar\Q_L$, i.e., the supremum over all accessible configurations $f$ of the length of the shortest solve from $f$, is at most $\omega_{\alpha+1}$. 
    \label{cor:edged_GN}
\end{corollary}

\begin{proof}
    This follows immediately from the previous theorem noting that every twist-finite basic sequence has length $< \omega_{\alpha+1}$. (Although there are trivial examples of twist-finite basic sequences having any desired length $< \omega_{\alpha+1}$.)
\end{proof}

It certainly is nice to know that every accessible legal configuration is in fact solvable and, moreover, that the convergent sequences here form a group. Nonetheless, this result is a far cry from anything resembling an \emph{algorithm} for solving a given configuration. Indeed, to implement the solution provided by the theorem, the solver needs to know the entire scramble, not merely the final scrambled configuration. This hardly counts as a ``solution'' as far as the Rubik's cube puzzle is concerned. Of course, we're being somewhat vague here, but there is an obvious sense in which this is result fails to provide a satisfactory solution method for the puzzle.

What's more, unlike the finite case, the mere existence of a sequence which will solve a given accessible configuration (together with our newfound understanding of the convergent sequences) is not enough for the solver to execute a ``brute force'' solution. Given any enumeration $s_\alpha$ of the twist finite sequences, the solver cannot apply (and subsequently undo using the inverse found above) all of them before this process itself fails to be twist-finite, hence diverges. So if the first solution only appears after that point, the solver will break the cube, so to speak, before finding it. Even if we imagine that the solver has tremendous intellectual capacity, capable of holding in her mind the entire infinite configuration of cardinality $\card L$ and simulating in her imagination the application of any given basic sequence (so that she may chose whether or not to execute any one of them on the actual cube she wishes to solve), there are $2^{\card L}$ many distinct twist-finite sequences (and just as many distinct accessible configurations). So there is a sense in which the ``brute force'' simulation requires vastly more resources than perhaps even our $\card L$-minded solver has access to. 

What one really wants, now that we know the puzzle is solvable in principle, is a uniform solution procedure that depends only on the given accessible configuration and which doesn't involve searching though up to $2^{\card L}$ many possible solutions. One might expect to be able to adapt algorithms for solving $\LLL N$ cubes but we know of no way of doing so in a twist-finite manner. 

However, it turns out that such an adaptation is possible in the countable edgeless case, where convergence over the solved configuration is a strictly weaker property than twist-finiteness.

\section{Efficient solution algorithm for the countable edgeless cube}
\label{sec:edgeless_algorithm}

The analogue of Lemma~\ref{lemma:edged_convergence} is not true in the edgeless case $\Q_L$. It's easy to produce basic sequences that are convergent over $f_\solved$ but not twist-finite, hence not universally convergent. In this sense, labels of cells are not generally preserved by basic sequences which converge over $f_\solved$ and so, at face value, it seems we don't get the decomposition into permutations of each cluster. But, thinking a little more carefully about the situation, one realizes that such a decomposition, albeit not necessarily a unique one, is possible anyway. The much more serious issue in the edgeless case is that finite iterates of a basic sequence convergent over $f_\solved$ need not converge over $f_\solved$ and so there is no obvious way to obtain a convergent sequence which achieves the inverse permutation $\big( \prod_C \pi_C \big)^{-1} = \big( \prod_C \pi_C \big)^{k-1}$, as we did previously.

\begin{observation}
    In the edgeless case $\Q_L$, the basic sequences convergent over $f_\solved$ are not closed under concatenation, and therefore do not form a submonoid of $\m_L$. In fact, there is such a sequence whose square is divergent over $f_\solved$.
\end{observation}

\begin{proof}
    Let $s$ be the sequence which first performs $\omega$ many quarter turns of the Front face then does a quarter turn in some $x$ layer. Clearly, $s$ is convergent over $f_\solved$ (since the initial face twists act trivially on this configuration) but $s^2$ is not. 
\end{proof}

At first blush, one might think this observation should disqualify the notion of convergence over $f_\solved$ as the right notion in the edgeless case, and suggest that we allow only universally convergent sequences, having already seen the robustness of that convergence notion. However, we feel this is the wrong conclusion. Certainly, it would have been convenient if these sequences formed a group (modulo $\sim$), but as far as scrambles are concerned, the Solver only really cares about the final \emph{configuration}, not the particular scramble that produced it, so long as there is \emph{some} sequence by which she might unscramble it, and the above observation certainly does not rule out this possibility. 

There are well known algorithms for the $\LLL N$ cube to solve a single cluster while leaving the rest unaffected and moreover, these algorithms require $O(1)$ many moves as $N \to \infty$. Since there are $O(N^2)$ clusters, solving each cluster serially yields an $O(N^2)$ algorithm. By cleverly parallelizing the cluster solves as in \cite{Demaine2011}, this can be improved to $O(N^2/\log N)$ (which they show is in fact optimal). Surprisingly, this modest improvement in the finite case generalizes to as dramatic a distinction as possible. Namely, the straightforward generalization of the serial solution does not even converge, but we can adapt the parallelized solution to a solution in the countable edgeless case in a mere $\omega^2$ many moves. 

This section closely mirrors \cite[\textsection4]{Demaine2011} with Lemmas~\ref{lemma:cluster_solution} and \ref{lemma:cluster_product_soln}, in particular, being direct generalizations to the infinite case of their reasoning. Aside from differences in notation and coordinate systems, we use \emph{cell cluster} to mean something slightly different than their \emph{cubie cluster} but the discrepancy is only witnessed by edge and corner clusters, so the notions coincide here. We mention this discrepancy to avoid possible confusion, particularly for reader who wishes to consider the edged case.

Recall that each non-center cluster has a unique representative $(x,y,+\infty)$ with $(x,y) \in L\times (\{0\} \cup L)$, that is, in the upper right quadrant of the Front face. Throughout this section, we will denote that quadrant by $\D$ and identify points $(x,y,+\infty) \in \D$ with their projections $(x,y)$. The images of this quadrant under the action of $G_L$ partition the non-center cells into 24 parts, the four quadrants of each face. This partition gives a correspondence between the cells of non-center clusters $C \leftrightarrow B$. We'll say that a cluster configuration is simply the restriction $f \rest C : C \to \Gamma$ of a configuration to a given cluster and we'll say that two non-center clusters have the same configuration if corresponding cells have the same color.

An intuitively clear but non-trivial fact is the following.

\begin{observation}
    In every legal accessible configuration, there are exactly four cells of each (non-$\NaC$) color in every non-center cluster.
    \label{obs:four}
\end{observation}

\begin{proof}
    As we've done many times, one argues by induction on the length of a sequences convergent over $f_\solved$ that each cluster has no more than four of each color. (Basic twists preserve this and if there were five of some color at some limit stage, there had to have already been five at some earlier stage.) But then since each non-center cluster has exactly 24 cells, and there are only six colors, there also cannot be fewer than four of each.
\end{proof}

A nice consequence of this is that in any accessible configuration of the edgeless cube, we may regard each non-center cluster configuration as an \emph{even} permutation of the solved configuration, i.e., for each cluster $C$ there is an even permutation $\pi_C \in A_C \cong A_{24}$ such that $f(c) = f_\solved(\pi_Cc)$ for all $c \in C$. The reason is that if $f = f_\solved\circ \pi$ on $C$ for some odd permutation $\pi \in S_C$, we can simply compose $\pi$ with a transposition that swaps two like-colored cells.\footnote{One can show that this holds also of the edge and corner cubies in the edged case but this situation is much more subtle because one needs to consider not just permutations of the cells but also simultaneously of the coupled cells. The coupled pairs/triples are unique so this transposing like-colors trick does not work. In this section we are only considering the edgeless cube anyway.}

\begin{lemma}
    For any given accessible configuration $f$ and any non-center cluster $C$, there is a finite basic sequence $s$ which solves $C$ and leaves all other clusters unchanged, i.e., if $g$ is the terminal configuration of $s|f$, then $g(c) = f_\solved(c)$ for each $c \in C$ and $g(c) = f(c)$ for all $c \not\in C$. Moreover, if $(\alpha,\beta,+\infty) \in \D$ for $\alpha > 0$ and $\beta \ge 0$ is the representative of $C$ in the Front Upper Right quadrant $\D$, then this sequence $s$ uses only face twists and slice twists in the $\pm\alpha$ and $\pm\beta$ layers.
    \label{lemma:cluster_solution}
\end{lemma}
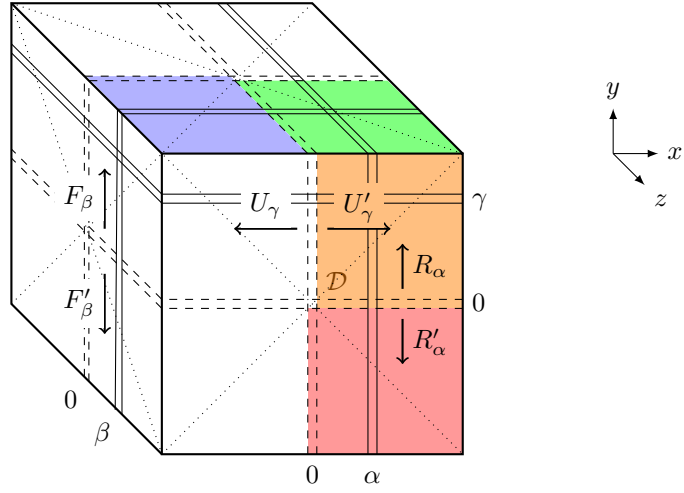
\begin{figure}[t]
    \begin{center}
        \begin{tikzpicture}[scale=2]
            \pgfmathsetmacro\r{.03}
            \pgfmathsetmacro\a{.4}
            \pgfmathsetmacro\b{.44}
            \pgfmathsetmacro\g{.7}
            \fill[orange!50] (1+\r,1-\r)  -- (1+\r,2) -- (2,2) -- (2,1-\r) -- cycle;
            \fill[red!40] (1-\r,1-\r) -- (2,1-\r) -- (2,0) -- (1-\r,0) -- cycle;
            \fill[green!50] (.5-\r/2,2.5-\r/2) -- (1.5,2.5-\r/2) -- (2,2) -- (1-\r,2) -- cycle;
            \fill[blue!30] (.5-3*\r/2,2.5+\r/2) -- (-.5,2.5+\r/2) -- (0,2) -- (1-\r,2) -- cycle;
            \draw[thick] (0,0)  -- (-1,1) -- (-1,3) -- (1,3) -- (2,2);
            \draw[thick] (-1,3) -- (0,2);
            \draw[thick] (0,0) -- (0,2) -- (2,2) -- (2,0) -- cycle;
            \node [above right, orange!50!black] at (1+\r,1+\r) {$\D$};

            \draw[dotted] (0,0) -- (2,2) -- (-1,3) -- cycle;
            \draw[dotted] (1,3) -- (0,2) -- (2,0);
            \draw[dotted] (-1,1) -- (0,2);

            \draw[dashed] (1+\r,0) -- (1+\r,2) -- (\r,3);
            \draw[dashed] (1-\r,0) -- (1-\r,2) -- (-\r,3);
            \draw (1+\a+\r,0) -- (1+\a+\r,2) -- (\a+\r,3);
            \draw (1+\a-\r,0) -- (1+\a-\r,2) -- (\a-\r,3);
            \node[below] at (1,0) {$\mathstrut0$};
            \node[below] at (1+\a,0) {$\mathstrut\alpha$};
            \draw[thick,->] (\a+1.2,1.1) -- (\a+1.2,1.4) node [midway,right] {$R_\alpha$};
            \draw[thick,->] (\a+1.2,.9) -- (\a+1.2,.6) node [midway,right] {$R_\alpha'$};

            \draw[dashed] (-.5-\r/2,.5+\r/2) -- (-.5-\r/2,2.5+\r/2) -- (1.5-\r/2,2.5+\r/2);
            \draw[dashed] (-.5+\r/2,.5-\r/2) -- (-.5+\r/2,2.5-\r/2) -- (1.5+\r/2,2.5-\r/2);
            \draw (-.5+\b/2-\r,.5-\b/2+\r/2) -- (-.5+\b/2-\r/2,2.5-\b/2+\r/2) -- (1.5+\b/2-\r/2,2.5-\b/2+\r/2);
            \draw (-.5+\b/2+\r/2,.5-\b/2-\r/2) -- (-.5+\b/2+\r/2,2.5-\b/2-\r/2) -- (1.5+\b/2+\r/2,2.5-\b/2-\r/2);
            \node[below left] at (-.5,.5) {$\mathstrut0$};
            \node[below left] at (-.5+\b/2,.5-\b/2) {$\mathstrut\beta$};

            \draw[dashed] (2,1+\r) -- (0,1+\r) -- (-1,2+\r);
            \draw[dashed] (2,1-\r) -- (0,1-\r) -- (-1,2-\r);
            \draw (2,1+\g+\r) -- (0,1+\g+\r) -- (-1,2+\g+\r);
            \draw (2,1+\g-\r) -- (0,1+\g-\r) -- (-1,2+\g-\r);
            \node[right] at (2,1) {$\mathstrut 0$};
            \node[right] at (2,1+\g) {$\mathstrut\gamma$};
            \draw[thick,->] (.9,1+\g-.2) -- +(-.42,0) node [above, midway, fill=white] {$\mathstrut U_\gamma$};
            \draw[thick,->] (1.1,1+\g-.2) -- +(.42,0) node [above, midway, fill=orange!50] {$\mathstrut \smash[b]{U_\gamma'}$};
            
            \draw[thick,->] (-.5+\b/2-.1, 1.5) -- +(0,.4) node [midway,left,fill=white] {$F_\beta$};
            \draw[thick,->] (-.5+\b/2-.1, 1.2) -- +(0,-.4) node [midway,left,fill=white] {$F_\beta'$};

            \draw[-latex] (3,2) -- +(.3,0) node [right] {$x$};
            \draw[-latex] (3,2) -- +(0,.3) node [above] {$y$};
            \draw[-latex] (3,2) -- +(.21,-.21) node [below right] {$z$};

        \end{tikzpicture}
    \end{center}
    \caption{Diagram indicating the alternate notations $R_\alpha = T_{x,\alpha}^{-1}$, $F_\beta = T_{z,\beta}^{-1}$, $U_\gamma = T_{y,\gamma}^{-1}$ for $\alpha,\beta,\gamma \ge 0$ and $R_\alpha', F_\beta', U_\gamma'$ are their respective inverses.}
    \label{fig:cube_diagram}
\end{figure}

\begin{proof}
    For readability, let's introduce the notations $R_\alpha,F_\beta,U_\gamma,R_\alpha',F_\beta',U_\gamma'$, akin to standard Rubik's cube notations, for twists in the right, upper, and front halves, as indicated in the diagram in Figure~\ref{fig:cube_diagram}. Precisely, for $\alpha,\beta,\gamma \ge 0$, we define $R_\alpha = T_{x,\alpha}^{-1}$, $F_\beta = T_{z,\beta}^{-1}$, and $U_\gamma = T_{y,\gamma}^{-1}$. One may think of these as clockwise quarter turns of the $\alpha,\beta,\gamma$ layer as seen from the Right, Front, or Upper face, respectively. $R_\alpha'$, $F_\beta'$, and $U_\gamma'$ are their respective inverses. 

    We claim that the sequence\footnote{To the author's knowledge, this sequence is well known in the speed cubing community. One secondary source attributed it to the speed cuber Ingo Sch\"utze but the author could not find a primary reference.}
    \begin{equation}
        \seq{R_\alpha, F_\beta', R_\alpha', F_\beta, U_{+\infty}', F_\beta', R_\alpha, F_\beta, R_\alpha', U_{+\infty}}
        \label{eq:3-cycle-seq}
    \end{equation}
    (read left-to-right) has the effect of rotating the cells of the $C(\beta,\alpha)$ cluster in the three quadrants highlighted in red, green, and blue in the diagram. That is, the sequence cycles $(\alpha,-\beta,+\infty)$, $(\alpha,+\infty,\beta)$, and $(-\beta,+\infty,\alpha)$ while leaving every other cell unchanged.

    To see that the sequence has the claimed effect, we break it down by cases, summarized in Table~\ref{tab:seq_effect} for $\alpha \ne \beta$. Table~\ref{tab:seq_effect_diag} gives the necessary modifications to Table~\ref{tab:seq_effect} for the case $\alpha = \beta$. Keep in mind that there are no edges, so the only cells affected by the Upper face twists $U_{+\infty}$ and $U_{+\infty}'$ are the cells in the Upper face.

    \begin{table}[t]
        \centering
        \makebox[\textwidth][c]{%
            \begin{tabular}{|l|l|l|l|}
                \hline
                face & cell & effective subsequence & reduced
                \\\hline\hline
                Down
                & $(x,-\infty,z)$; $x \ne \alpha,\,z\ne\beta$   & $\varepsilon$ & $\varepsilon$
                \\& $(\alpha,-\infty,z)$; $z \ne \beta$       & $R_\alpha,R_\alpha',R_\alpha,R_\alpha'$   & $\varepsilon$
                \\& $(x,-\infty,\beta)$; $x \ne \alpha$       & $F_\beta',F_\beta,F_\beta',F_\beta$ & $\varepsilon$
                \\& $(\alpha,-\infty,\beta)$       & $R_\alpha,R_\alpha',F_\beta,F_\beta',R_\alpha,R_\alpha'$ & $\varepsilon$
                \\\hline
                Left
                & $(-\infty,y,z)$; $z\ne\beta$   & $\varepsilon$ & $\varepsilon$
                \\& $(-\infty,y,\beta)$; $y \ne -\alpha$  & $F_\beta',F_\beta,F_\beta',F_\beta$ & $\varepsilon$
                \\& $(-\infty,-\alpha,\beta)$  & $F_\beta',R_\alpha',R_\alpha,F_\beta$ & $\varepsilon$
                \\\hline
                Right
                & $(+\infty,y,z)$; $z\ne\beta$   & $\varepsilon$ & $\varepsilon$
                \\& $(+\infty,y,\beta)$; $y \ne -\alpha$  & $F_\beta',F_\beta,F_\beta',F_\beta$ & $\varepsilon$
                \\& $(+\infty,-\alpha,\beta)$  & $F_\beta',R_\alpha',R_\alpha,F_\beta$ & $\varepsilon$
                \\\hline
                Back
                & $(x,y,-\infty)$; $x\ne\alpha$   & $\varepsilon$ & $\varepsilon$
                \\& $(\alpha,y,-\infty)$; $y\ne-\beta$   & $R_\alpha,R_\alpha',R_\alpha,R_\alpha'$ & $\varepsilon$
                \\& $(\alpha,-\beta,-\infty)$   & $R_\alpha,F_\beta',F_\beta,F_\beta',F_\beta,R_\alpha'$ & $\varepsilon$
                \\\hline
                Front
                & $(x,y,+\infty)$; $x\ne\alpha$   & $\varepsilon$ & $\varepsilon$
                \\& $(\alpha,y,+\infty)$; $y\ne-\beta$   & $R_\alpha,R_\alpha',R_\alpha,R_\alpha'$ & $\varepsilon$
                \\& $(\alpha,-\beta,+\infty)$  & $R_\alpha,F_\beta',F_\beta,U_{+\infty}',U_{+\infty}$ & $R_\alpha$
                \\\hline
                Up
                & $(x,+\infty,z)$; $x\ne\pm\alpha,\pm\beta$, $z\ne\alpha,\beta$   & $U_{+\infty}',U_{+\infty}$ & $\varepsilon$
                \\& $(x,+\infty,\alpha)$; $x\ne\pm\alpha,\pm\beta$ & $U_{+\infty}',R_\alpha,R_\alpha',U_{+\infty}$ & $\varepsilon$
                \\& $(x,+\infty,\beta)$; $x\ne\pm\alpha,\pm\beta$ & $F_\beta',F_\beta,U_{+\infty}',U_{+\infty}$ & $\varepsilon$
                \\& $(\alpha,+\infty,z)$; $z\ne\alpha,\beta$   & $R_\alpha,R_\alpha',U_{+\infty}',U_{+\infty}$ & $\varepsilon$
                \\& $(\alpha,+\infty,\alpha)$  & $R_\alpha,R_\alpha',U_{+\infty}',R_\alpha,R_\alpha',U_{+\infty}$ & $\varepsilon$
                \\& $(\alpha,+\infty,\beta)$  & $R_\alpha,R_\alpha',F_\beta,F_\beta',R_\alpha,R_\alpha',U_{+\infty}$ & $U_{+\infty}$
                \\& $(\beta,+\infty,z)$; $z \ne \alpha,\beta$ & $U_{+\infty}',U_{+\infty}$ & $\varepsilon$
                \\& $(\beta,+\infty,\alpha)$ & $U_{+\infty}',R_\alpha,R_\alpha',U_{+\infty}$ & $\varepsilon$
                \\& $(\beta,+\infty,\beta)$ & $F_\beta',F_\beta,U_{+\infty}',U_{+\infty}$ & $\varepsilon$
                \\& $(-\beta,+\infty,z)$; $z \ne \alpha,\beta$ & $U_{+\infty}',F_\beta',F_\beta,U_{+\infty}$ & $\varepsilon$
                \\& $(-\beta,+\infty,\alpha)$ & $U_{+\infty}',F_\beta',F_\beta,R_\alpha'$ & $U_{+\infty}',R_\alpha'$
                \\& $(-\beta,+\infty,\beta)$ & $F_\beta',F_\beta, U_{+\infty}',F_\beta',F_\beta,U_{+\infty}$ & $\varepsilon$
                \\& $(-\alpha,+\infty,z)$; $z \ne \alpha,\beta$ & $U_{+\infty}',U_{+\infty}$ & $\varepsilon$
                \\& $(-\alpha,+\infty,\alpha)$;  & $U_{+\infty}',R_\alpha,R_\alpha',U_{+\infty}$ & $\varepsilon$
                \\& $(-\alpha,+\infty,\beta)$;  & $F_\beta',F_\beta,U_{+\infty}',U_{+\infty}$ & $\varepsilon$
                \\\hline
            \end{tabular}%
        }
        \caption{Effect of the sequence $\seq{R_\alpha,F_\beta',R_\alpha',F_\beta,U_{+\infty}',F_\beta',R_\alpha',F_\beta,R_\alpha',U_{+\infty}}$ for $\alpha \ge0$ and $\beta> 0$ and $\alpha \ne \beta$ on each cell by cases. Here $\varepsilon$ denotes the empty sequence. Only the last three rows ($x = -\alpha$ in the Upper face) do not apply if $\alpha = 0$.}
        \label{tab:seq_effect}
    \end{table}
    \begin{table}[t]
        \centering
        \begin{tabular}{|l|l|l|l|}
            \hline
            face & cell & effective subsequence & reduced
            \\\hline\hline
            Front
            & $(x,y,+\infty)$; $x\ne\alpha$   & $\varepsilon$ & $\varepsilon$
            \\& $(\alpha,y,+\infty)$; $y\ne-\alpha$   & $R_\alpha,R_\alpha',R_\alpha,R_\alpha'$ & $\varepsilon$
            \\& $(\alpha,-\alpha,+\infty)$  & $R_\alpha,F_\beta',F_\beta,U_{+\infty}',R_\alpha,R_\alpha',U_{+\infty}$ & $R_\alpha$
            \\\hline
            Up
            & $(x,+\infty,z)$; $x\ne\pm\alpha$, $z\ne\alpha$   & $U_{+\infty}',U_{+\infty}$ & $\varepsilon$
            \\& $(x,+\infty,\alpha)$; $x\ne\pm\alpha$ & $F_\alpha',F_\alpha,U_{+\infty}',R_\alpha,R_\alpha',U_{+\infty}$ & $\varepsilon$
            \\& $(\alpha,+\infty,z)$; $z\ne\alpha$   & $R_\alpha,R_\alpha',U_{+\infty}',U_{+\infty}$ & $\varepsilon$
            \\& $(\alpha,+\infty,\alpha)$  & $R_\alpha,R_\alpha',F_\alpha,F_\alpha',R_\alpha,R_\alpha',U_{+\infty}$ & $U_{+\infty}$
            \\& $(-\alpha,+\infty,z)$; $z \ne \alpha$ & $U_{+\infty}',F_\alpha',F_\alpha,U_{+\infty}$ & $\varepsilon$
            \\& $(-\alpha,+\infty,\alpha)$;  & $F_\alpha',F_\alpha',U_{+\infty}',F_\alpha',F_\alpha,R_\alpha'$ & $U_{+\infty}',R_\alpha'$
            \\\hline
        \end{tabular}
        \caption{Effect of the sequence $\seq{R_\alpha,F_\alpha',R_\alpha',F_\alpha,U_{+\infty}',F_\alpha',R_\alpha',F_\alpha,R_\alpha',U_{+\infty}}$ for $\alpha > 0$ on each cell of the Front and Upper faces cases. Table~\ref{tab:seq_effect} applies verbatim (with $\beta = \alpha$) for the other four faces.}
        \label{tab:seq_effect_diag}
    \end{table}

    Now, one can preform the analogous sequence thinking of any of the six faces as the front and any of the four adjacent faces as the upper one. Thus, we obtain 24 distinct 3-cycles of the face quadrants so that for each such 3-cycle $\sigma$ and each (non-center) cluster $C(\beta,\alpha)$, we can permute the cluster by $\sigma$ in finitely many twists, using only $\pm\alpha,\pm\beta$ slice twists and face twists. 

    Fixing any enumeration $\D_1,\dots,\D_{24}$ of the face quadrants, we may think of these 3-cycles as permutations of $\{1,\dots,24\}$. One verifies (using the GAP software package \cite{GAP2024}) that these 24 3-cycles generate $A_{24}$. As noted earlier, we may assume that in any accessible legal configuration, each cluster configuration is an even permutation of the solved configuration, thus completing the proof.
\end{proof}

\begin{remark}
    The above lemma and its proof are essentially identical to \cite[Lemma~8]{Demaine2011}. However, in the author's opinion, Tables~\ref{tab:seq_effect} and \ref{tab:seq_effect_diag} provide a more convincingly thorough and easily verified (if unwieldy) account of the situation than the prosaic one offered in \cite{Demaine2011}. Moreover, Demaine et al.\ are somewhat unclear regarding cross and diagonal clusters (they can afford not to worry about these clusters anyway, whereas we have no such luxury).
\end{remark}

The parity Observation~\ref{obs:four} and the single-cluster solution method Lemma~\ref{lemma:cluster_solution} apply to the ordinary $\LLL N$ Rubik's cube and the proofs are actually a little easier in our case for the lack of edges and corners, which complicate matters slightly. 

Where the infinite case becomes problematic is that there does not seem to be a way to use the move sequences given by the previous lemma to solve each cluster in series without leading to a divergent configuration. For example, consider the configuration $f = T_{x,\alpha}f_\solved$ obtained by twisting only the $x = \alpha$ layer by a quarter-turn for some $\alpha \in L$. Suppose we ignore the obvious solution in the interest of being systematic. We need to solve all infinitely many clusters intersecting the $x=\alpha$ column of the Front face, that is each $C( (\alpha,y,+\infty) )$ for $y \in L^\dagger$. If we try to solve them one-by-one, then after attempting to solve $\omega$ many clusters (even if this is all of them), it seems like we may have changed some cell colors, in say the $z = \alpha$ layer unboundedly often, and so the sequence diverges at stage $\omega$. We emphasize that this problem occurs even if $L$ is countable. Although it might seem like one is making steady progress before stage $\omega$, approaching having solved every cluster, this is not enough. One must also ensure that in solving the later clusters, one does not infinitely often temporarily upset any cell one has already solved, or any other cell for that matter. 

We need to be more clever and the strategy is evident. Namely, we need to highly parallelize these cluster solution sequences in such a way that we obtain a convergent solve. We do not yet know how to do this in the uncountable case, but we have such a parallel method when $L$ is countable and we should like to explicate this now. Curiously, the method resembles some traditional Rubik's cube algorithms, proceeding from the cross clusters ``outward''. Indeed, one may have guessed that a solution might exist along these lines but hopefully by now, the reader appreciates the obstacles to na\"ively implementing such a procedure in a convergent manner.

We will need the machinery of cluster move sequences introduced by Demaine~et al.~in \cite{Demaine2011} and we shall reformulate the method to put it into our notation and, more importantly, to be sure it generalizes as claimed to the infinite case. Moreover, for their purposes they could afford to solve the cross and diagonal clusters in series whereas we cannot.

Let us introduce some terminology (some of which we have already be using informally). 
A \emph{face move} is a twist $T_{i,\pm\infty}^{\pm1}$, that is, a quarter- or reverse quarter-turn of any one face. They come in twelve \emph{types} (two directions for each of six faces) and, as a convenience, we define a thirteenth type which applies the identity. If $a$ is the type of face move, we write $F_a$ for the twist itself (not to be confused with the $F_\beta$ notation we used only in the proof of the previous lemma). A \emph{slice move} is any quarter- or reverse quarter turn basic twist $T_{i,\alpha}^{\pm1}$ for $\alpha \ne \pm\infty$. Given a index $0 \le \alpha < +\infty$, there are again twelve types of slice moves, three axes, two choices of sign ($\alpha$ and $-\alpha$), and two directions of rotation, and again, we define a thirteenth type which applies the identity. If $a$ is the type of slice move and $\alpha$ is its index, we write $S_{a,\alpha}$ for the twist itself. The moves $F_a$ and $S_{a,\alpha}$ for all types and indices are unique except that there are two types for each cross slice $S_{a,0}$ since $0 = -0$. 

A \emph{cluster move sequence} is a (finite) sequence of twists of the form 
\[
    \seq{S_{a_1,x},S_{b_1,y},F_{c_1},\dots,S_{a_k,x},S_{b_y,y},F_{c_k}}
\]
specified by a cluster representative $(x,y) \in \D$ and three (finite) type sequences: slice types $a_1,\dots,a_k$, slice types $b_1,\dots,b_k$, and face types $c_1,\dots,c_k$. A \emph{cluster move solution} for a cluster configuration $h$ is a cluster move sequence with the following properties:
\begin{enumerate}[label=(\Roman*)]
    \item For any $(x,y) \in \D$, if cluster $C(x,y)$ is in configuration $h$, then it can be solved (i.e., put in configuration $f_\solved \rest C(x,y)$) by applying the cluster move sequence $\seq{S_{a_1,x},S_{b_1,y},F_{c_1},\dots,S_{a_k,x},S_{b_k,y},F_{c_k}}$
    \item If $x \ne y$ and $y \ne 0$ (i.e., if $C(x,y)$ is not a diagonal or cross cluster), then the aforementioned cluster move sequence leaves $C(y,x)$ unchanged.
    \item All three of the following basic sequences leave all clusters unchanged:
        \begin{align*}
            & \seq{S_{a_1,x}, F_{c_1}, S_{a_2,x}, F_{c_2}\dots, S_{a_k,x}, F_{c_k}},
            \\& \seq{S_{b_1,y}, F_{c_1}, S_{b_2,y}, F_{c_2}\dots, S_{b_k,y}, F_{c_k}},
            \\& \seq{F_{c_1},F_{c_2}, \dots, F_{c_k}},
        \end{align*}
\end{enumerate}
When we say here that a sequence leaves a cluster unchanged, we mean only that it returns it to its original configuration, not that it does not temporarily disturb it. For convergence considerations, we will need to keep careful track of all disturbances.
Note that these various basic sequences are all (universally) convergent since they are finite. 

The sequence \eqref{eq:3-cycle-seq} and its rotated analogues are cluster move sequences (involving some identity face moves) and furthermore, the subsequences corresponding to property (III) above all act trivially. It follows immediately that the sequence provided by the proof of Lemma~\ref{lemma:cluster_solution} is in fact a cluster move solution for $C(\beta,\alpha)$. The key parallelization step is provided by the following lemma.

\begin{lemma}
    Suppose $X \subset L$ and $Y \subset \{0\}\times L$ (so that $X \times Y \subset \D$) with $X$ and $Y$ disjoint. If all clusters intersecting $X \times Y$ have the same configuration, then there is a twist-finite basic sequence of ordinal length $ \le (\card X + \card Y + 1)\cdot k$ for some finite $k$ only involving face moves and slice moves with index in $X \cup Y$ which when applied to the given configuration results in a configuration in which all clusters intersecting $X \times Y$ are solved and which agrees with the original configuration except possibly on clusters intersecting $(X\times Y) \cup (X\times X) \cup (Y \times Y)$. The final configuration also agrees with the original on the center cluster.
    \label{lemma:cluster_product_soln}
\end{lemma}

\begin{proof}
    This argument is essentially due to Demaine et al.~\cite{Demaine2011}. We reformulate it here to be explicit about how it generalizes to the infinite case.

    Let $h$ be the common configuration of all the clusters in $X\times Y$. By the previous lemma, these clusters admit cluster move solutions with common type sequences $a_1,\dots,a_k$, $b_1,\dots,b_k$, and $c_1,\dots,c_k$ (depending only on the configuration $h$).

    Let $s\concat t$ denote concatenation of (transfinite) sequences $s$ and $t$, read left-to-right (i.e., $s$ then $t$). Define for each $1\le i \le k$ the sequence
    \[
        s_i = \seq{S_{a_i,x} : x \in X}\concat \seq{S_{b_i,y} : y \in Y}\concat \seq{F_{c_i}}.
    \]
    Note that slice moves of the same type commute so we may take any well-order on $X$ and $Y$ in the definition of $s_i$. Being frugal, we might as well take $s_i$ to have length $\le \card X + \card Y + 1$ (interpreted in the ordinal sense). $s_i$ might be strictly shorter than $\card X + \card Y + 1$ if it involves identity ``moves''. Note also that $s_i$ is twist-finite. Consider the sequence 
    \[
        s = s_1\concat s_2 \concat \cdots \concat s_k.
    \]
    of length at most $(\card X + \card Y + 1)\cdot k$. Note that any cell in cluster $C(\alpha,\beta)$ can only possibly be affected non-trivially by face moves and by slice moves of index $\alpha$ or $\beta$ (that is, twists of some $\pm\alpha$ or $\pm\beta$ layer). Hence, for each $(x,y) \in X\times Y$, the subsequence of $s$ which affects it is
    \[
        \seq{S_{a_1,x},S_{b_1,y},F_{c_1},\dots,S_{a_k,x},S_{b_k,y},F_{c_k}}
    \]
    which, by construction, is precisely a cluster move solution for this cluster. (This is where we need the assumption that $X$ and $Y$ are disjoint, otherwise there might be interference between the $a_i$ and $b_j$ type moves.) This same subsequence will affect the cluster of $(y,x)$. If $y \ne 0$, then this sequence leaves that cluster unchanged by property (II) of cluster move solutions. (If $y = 0$ then $(x,0)$ and $(0,x)$ belong to the same cross cluster, which will be solved.) 

    The clusters intersecting $X\times X$ or $Y\times Y$ may or may not be affected, but the result we are after here puts no restriction on these clusters. It remains to consider the clusters whose $\D$ representative has at most one coordinate in $X\cup Y$. Suppose $(x,w) \in \D$ with $x \in X$ and $w \not\in X \cup Y$. Then the subsequence of $s$ which affects the clusters of $(x,w)$ and $(w,x)$ is $\seq{S_{a_1,x},F_{c_1},\dots,S_{a_k,x},F_{c_k}}$ which has no ultimate effect on any cluster by property (III) of move solutions. Similarly, if $(w,y) \in \D$ with $y \in Y$ and $w \not\in X \cup Y$, the subsequence of $s$ affecting the clusters of $(w,y)$ and $(y,w)$ is $\seq{S_{b_1,y},F_{c_1},\dots,S_{b_k,y},F_{c_k}}$ which does not ultimately affect any cluster again by property (III). Lastly, for every $(v,w) \in \D$ with $v,w \not\in X \cup Y$, the subsequence of $s$ affecting $(v,w)$ is $\seq{F_{c_1},\dots,F_{c_k}}$, which again has no ultimate effect on any cluster.
\end{proof}

\begin{remark}
    Thinking carefully about the arguments in this section, one realizes that the finite $k$ in the previous lemma is universal, it does not depend on the concerned clusters nor on their configuration.
\end{remark}

From here, our argument deviates from that of Demaine~et al. Their algorithm simply does not generalize to a convergent one. Everything we have established so far is completely general. Only now will we restrict our attention to the countable case $\card L = \aleph_0$. 

\begin{theorem}
    On the countable edgeless cube $\Q_L$ for $\card L = \aleph_0$, every accessible configuration can be solved and in at most $\omega^2$ many moves.
    \label{thm:edgeless_countable_solution}
\end{theorem}

\begin{proof}
    Without loss of generality, we take $L = \{1,2,3,\dots\}$ to be the positive integers. It will be convenient to endow $L$ with the usual natural number order (so that $L^\dagger$ looks like $\Z$).

    Let $f_0$ be any accessible configuration of $\Q_L$. The strategy is as follows. We will apply a basic sequence of twists in $\omega$ many stages. Each stage will involve up to $\omega\cdot n + k$ many twists for some finite $n,k$, therefore resulting in $\omega^2$ many twists total. Before the final stage $\omega^2$ we will have only used each particular basic twist finitely many times so the only question of convergence is at the ultimate stage $\omega^2$. We'll argue that the color of each cell $c$ stabilizes on $f_\solved(c)$ before stage $\omega^2$. 

    As a preliminary step, in finitely many moves, we can solve the center cluster. (Alternatively, since the center cells cannot change there relative positions, we can simply define our coordinates so that the center cluster is already solved, which amounts to a global rotation.)

    Recall cluster configurations correspond to permutations in $S_{24}$ (via permutation of the face quadrants), in particular, there are only finitely many cluster configurations, at most $m = \card{S_{24}} = 24!$. (We can improve this to $24!/4!^6$ by considering permutations of the four like-colored cells of each color in each cluster, but it doesn't matter here.) A high-level overview of the first stages we are about to describe is represented in Figure~\ref{fig:schematic}. 

    In stage 0, we appeal to Lemma~\ref{lemma:cluster_product_soln} to solve all the cross clusters as follows. For each of the $m$ cluster configurations $h$, let $X_h \subset L$ be the subset of $L$ consisting of all (non-zero) $x$ for hich the cluster of $(x,0,+\infty)$ is in configuration $h$. Applying Lemma~\ref{lemma:cluster_product_soln} to $X_h \times \{0\}$, we solve all the cross clusters in configuration $h$ in $\le (\card{X_h} + 2) \cdot k \le \omega\cdot k + 2$ many moves. Doing this successively for each cluster configuration $h$, we solve all non-center cross clusters in $\le (\omega\cdot k + 2)\cdot m = \omega \cdot km + 2$ many moves. Moreover, this particular sequence returns the center cluster to the solved configuration.

    In stage 1, having solved the center and cross clusters, we move on to their neighboring clusters, that is, those intersecting the $\pm1$ layers. For each cluster configuration $h$, let $X_h \subset L\setminus\{1\}$ be the set of $x > 1$ for which the cluster of $(x,1,+\infty)$ is in configuration $h$. We solve all clusters intersecting $X_h \times \{1\}$ in $\le \omega\cdot k + 2$ many moves and looping over all cluster configurations $h$, we solve all $(L\setminus \{1\}) \times \{1\}$ clusters in $\le \omega\cdot km + 2$ many moves. (Note, we solved the $(0,1)$ cluster in the previous stage.) We can then do the same thing in the $x=1$ position, defining $Y_h$ analogously to $X_h$, etc. In another $\le \omega\cdot km +2$ many moves, we have solved all clusters in the $\pm1$ layers except possibly the diagonal cluster $C(1,1)$. We can now solve this cluster individually in finitely many moves by Lemma~\ref{lemma:cluster_solution}. So stage 1 requires in total $\le (\omega\cdot km + 2) + (\omega\cdot km+2 ) + \omega \le \omega\cdot (2km + 1)$ many moves.

    Let us highlight a few important points. Altogether, we have only used each particular basic twist finitely many times because we applied Lemma~\ref{lemma:cluster_product_soln} finitely many times ($\le m$ times in stage 0 and $\le 2m$ times in stage 1) and Lemma~\ref{lemma:cluster_solution} once, applying altogether only finitely many twist-finite sequences. At the end of stage 1, every cluster meeting the $0$ and $\pm1$ layers is in the solved configuration. Crucially to what follows, notice that stage 1 did not require any center layer slice moves $S_{a,0}$ and center cells are invariant under face twists. Thus, the color of the center cells never changed during stage 1. Keep this in mind as we think through the rest of the stages. Let's go through one more stage in detail to get a feel for what happens in general.

    In stage 2, we define $X_h \subset L \setminus \{1,2\}$ for each cluster configuration $h$ as the set of those $x > 2$ with the cluster of $(x,2,+\infty)$ in configuration $h$. By looping through the $m$ many sets of like-configured clusters $X_h \times \{2\}$, we solve all the $(L\setminus\{1,2\})\times \{2\}$ clusters in $\le \omega\cdot km + 2$ many moves. We have already solved $\{1\}\times \{2\}$. We then do the same for the $\{2\} \times (L\setminus \{1,2\})$ clusters, again having already solved the $\{2\}\times\{1\}$ cluster. Finally, we solve the diagonal $(2,2)$ cluster. Again, this stage uses each twist only finitely many times, uses $\le \omega\cdot (2km+1)$ many twists total, and at its conclusion, we the clusters in the $0,\pm1,\pm2$ layers are solved. Now, the key feature is that stage 2 does not require any slice moves of the $0$ or $\pm1$ layers and at the start of the stage, the cells in the $[-1,1]^2$ center of each face were already solved, hence invariant under face twists. So the cells in this center square of each face never change color during stage 2.

    And so on. In stage $n$, we solve all the (not yet solved) clusters in the $\pm n$ layers via a twist-finite sequence of length $\le \omega\cdot (2km+1)$ which does not include any slice moves of the $\pm \ell$ layers for $\ell < n$. Therefore, the already solved $[-(n-1),n-1]^2$ core of each face, which is invariant under face twists, is constant colored throughout stage $n$.

    Each cell $c = (x,y,z)$ is contained in the solved invariant core of its face forever after stage $n$ where $n$ is its maximum finite coordinate among $\pm x, \pm y, \pm z$, and its colors stabilizes on $f_\solved(c)$. Therefore, the sequence of configurations converges to $f_\solved$ at stage $\omega^2$, if not sooner.
\end{proof}

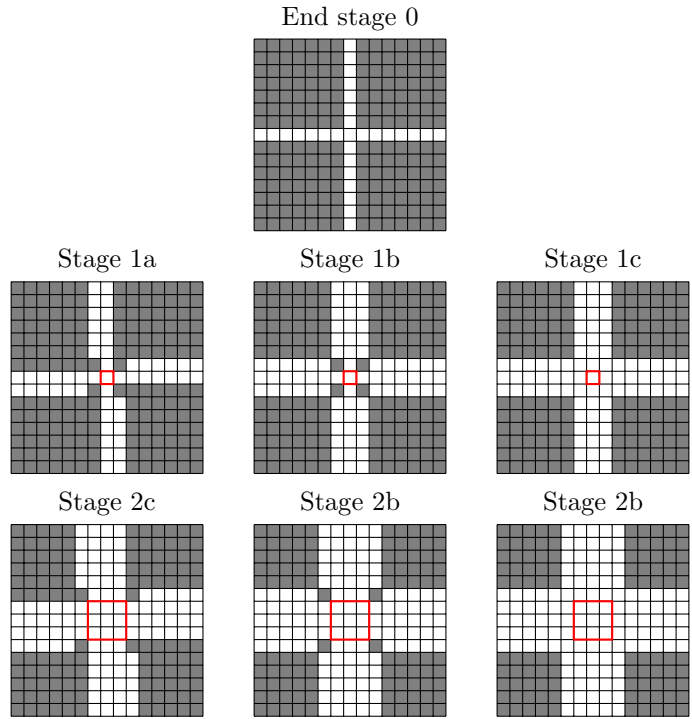
\begin{figure}[th]
    \centering
    \begin{tikzpicture}[scale=.17]
        \pgfmathsetmacro\r{7}
        \pgfmathsetmacro\xs{2*\r+5}
        \pgfmathsetmacro\ys{-2*\r-5}
        \begin{scope}[shift={(\xs,0)}]
            \fill[gray] (-\r,-\r) rectangle (\r+1,\r+1);
            \fill[white] (0,-\r) rectangle (1,\r+1);
            \fill[white] (-\r,0) rectangle (\r+1,1);
            \draw (-\r,-\r) grid (\r+1,\r+1);
            \node [above] at (.5,\r+1) {End stage 0};
        \end{scope}
        
        \begin{scope}[shift={(0,\ys)}]
            \fill[gray] (-\r,-\r) rectangle (\r+1,\r+1);
            \fill[white] (0,-\r) rectangle (1,\r+1);
            \fill[white] (-\r,0) rectangle (\r+1,1);
            \fill[white] (2,1) rectangle (\r+1,2);
            \fill[white] (-1,0) rectangle (-\r,-1);
            \fill[white] (-1,2) rectangle (0,\r+1);
            \fill[white] (1,-1) rectangle (2,-\r);
            \draw (-\r,-\r) grid (\r+1,\r+1);
            \draw[red, thick] (0,0) rectangle (1,1);
            \node [above] at (.5,\r+1) {Stage 1a};
        \end{scope}
        \begin{scope}[shift={(\xs,\ys)}]
            \fill[gray] (-\r,-\r) rectangle (\r+1,\r+1);
            \fill[white] (-1,-\r) rectangle (2,\r+1);
            \fill[white] (-\r,-1) rectangle (\r+1,2);
            \fill[gray] (1,1) rectangle (2,2);
            \fill[gray] (1,0) rectangle (2,-1);
            \fill[gray] (0,1) rectangle (-1,2);
            \fill[gray] (0,0) rectangle (-1,-1);
            \draw (-\r,-\r) grid (\r+1,\r+1);
            \draw[red, thick] (0,0) rectangle (1,1);
            \node [above] at (.5,\r+1) {Stage 1b};
        \end{scope}
        \begin{scope}[shift={(2*\xs,\ys)}]
            \fill[gray] (-\r,-\r) rectangle (\r+1,\r+1);
            \fill[white] (-1,-\r) rectangle (2,\r+1);
            \fill[white] (-\r,-1) rectangle (\r+1,2);
            \draw (-\r,-\r) grid (\r+1,\r+1);
            \draw[red, thick] (0,0) rectangle (1,1);
            \node [above] at (.5,\r+1) {Stage 1c};
        \end{scope}
        \begin{scope}[shift={(0,2*\ys)}]
            \fill[gray] (-\r,-\r) rectangle (\r+1,\r+1);
            \fill[white] (-1,-\r) rectangle (2,\r+1);
            \fill[white] (-\r,-1) rectangle (\r+1,2);
            \fill[white] (3,2) rectangle (\r+1,3);
            \fill[white] (-1,3) rectangle (-2,\r+1);
            \fill[white] (-2,-1) rectangle (-\r,-2);
            \fill[white] (2,-2) rectangle (3,-\r);
            \draw (-\r,-\r) grid (\r+1,\r+1);
            \draw[red, thick] (-1,-1) rectangle (2,2);
            \node [above] at (.5,\r+1) {Stage 2c};
        \end{scope}
        \begin{scope}[shift={(\xs,2*\ys)}]
            \fill[gray] (-\r,-\r) rectangle (\r+1,\r+1);
            \fill[white] (-2,-\r) rectangle (3,\r+1);
            \fill[white] (-\r,-2) rectangle (\r+1,3);
            \fill[gray] (2,2) rectangle (3,3);
            \fill[gray] (2,-1) rectangle (3,-2);
            \fill[gray] (-1,2) rectangle (-2,3);
            \fill[gray] (-1,-1) rectangle (-2,-2);
            \draw (-\r,-\r) grid (\r+1,\r+1);
            \draw[red, thick] (-1,-1) rectangle (2,2);
            \node [above] at (.5,\r+1) {Stage 2b};
        \end{scope}
        \begin{scope}[shift={(2*\xs,2*\ys)}]
            \fill[gray] (-\r,-\r) rectangle (\r+1,\r+1);
            \fill[white] (-2,-\r) rectangle (3,\r+1);
            \fill[white] (-\r,-2) rectangle (\r+1,3);
            \draw (-\r,-\r) grid (\r+1,\r+1);
            \draw[red, thick] (-1,-1) rectangle (2,2);
            \node [above] at (.5,\r+1) {Stage 2b};
        \end{scope}
    \end{tikzpicture}
    \caption{Schematic of the initial stages of the solve in the white face. In each case, the pattern is understood to continue in the obvious way. Gray cells represent arbitrary colors here. Each stage $n$a involves finitely many highly parallelized cluster solution subroutines for $X_h\times\{n\}$ clusters (applications of Lemma~\ref{lemma:cluster_product_soln}). Stage $n$b is analogous for $\{n\}\times Y_h$ clusters.  Stage $n$c solves the individual diagonal $(n,n)$ cluster. The red outlined box indicated the region invariant through stage $n$ (and forever after). The meat of the proof is in the existence of these twist-finite sequences which solve infinitely many clusters at once and doing so without ever upsetting the increasing core of cells at the center of each face.} 
    \label{fig:schematic}
\end{figure}

\begin{corollary}
    God's number for the countable edgeless cube $\Q_L$ with $\card L = \aleph_0$ is at most $\omega^2$.
    \label{cor:GN_edgeless}
\end{corollary}

In our opinion, these results provide a very satisfactory formulation of and solution to the vague ``$\LLL \Z$'' Rubik's cube puzzle. Let us highlight a few notable features. This argument does not work for the (countable) edged cube $\bar\Q_L$. The sequence of twists produced by this method is \emph{almost}, but not quite, twist-finite. It uses each slice move $S_{a,\alpha}$ only finitely many times but generally uses $\omega$ many face moves (of which there are only finitely many distinct instances). So we know this sequence cannot converge on the edged cube but, in fact, it will converge pointwise to the solved configuration except on the edge and corner cells.

This argument also does not generalize in any obvious way to the uncountable edgeless case. One major obstacle is the Cartesian product structure imposed by Lemma~\ref{lemma:cluster_product_soln} on the clusters whose solution sequences we can parallelize. We circumvented this issue in the countable case by parallelizing along one row or column at a time, trivializing the product structure. Without a deeper understanding of the accessible configurations themselves in the edgeless case, it seems extremely challenging to make progress.

Still, we can push Theorem~\ref{thm:edgeless_countable_solution} a little further. Say a configuration $f$ of the edgeless cube $\Q_L$ is \emph{standard} if each non-center cluster has exactly four of each non-$\NaC$ color and the configuration of the center cluster (if it exists) is a global rotation of the solved center cluster configuration. Standardness is obviously a necessary condition for solvability, we now show that it is also sufficient, plus a little more.

\begin{lemma}
    For the countable edgeless cube, if $f_0$ standard configuration and $f$ is any standard configuration which, additionally, is invariant under all quarter-turn face twists, then $f$ is accessible from $f_0$ in at most $\omega^2$ many moves. 
    \label{lemma:std_access}
\end{lemma}

\begin{proof}
    It follows from Observation~\ref{obs:four} and the (easily verified) fact that convergent sequences cannot change the relative positions of the center cell colors that every legal configuration $f_0$ accessible from the solved configuration is standard. (Although we do not know if the converse is true.)

    Let $f$ be a standard configuration invariant under face twists. For each non-center cluster $C$, we may assume its configuration $f_0 \rest C$ is an even permutation of $f \rest C$, i.e., that there is some $\pi_C \in A_C \cong A_{24}$ such that $f(c) = f_0(\pi_C c)$ for each $c \in C$. Again, this is because if $\pi_C$ were odd, we could compose it with a transposition swapping two like-colored cells which exist since the configurations are standard. 

    Arguing exactly as in Theorem~\ref{thm:edgeless_countable_solution} with $f$ playing the role of $f_\solved$ as the target configuration, we can in $\omega^2$ many moves, in effect apply the permutation $\pi_C^{-1}$ to each cluster. Note, now we are parellalizing the cluster move sequences based on the permutations $\pi_C$ rather than the cluster configurations $h$, these notions were same when $f$ was $f_\solved$. The face-twist invariance of $f$ ensures convergence at stage $\omega^2$. 
\end{proof}

This result lets us paint a much richer picture of the space of legal configurations of $\Q_{\aleph_0}$ connected to $f_\solved$ by accessibility. We summarize what we know in Figure~\ref{fig:config_space} and make this precise in the following theorem.

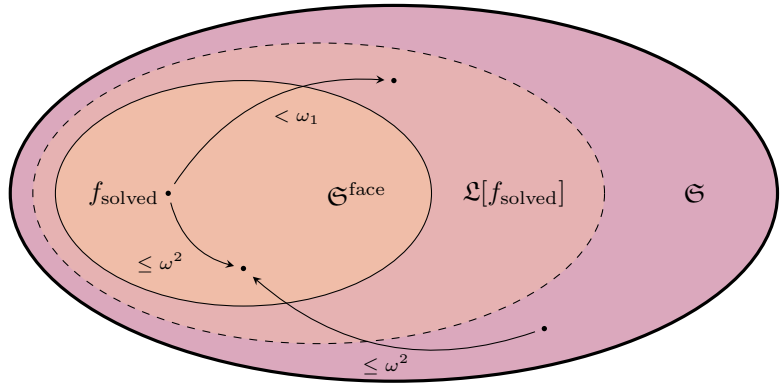
\begin{figure}[th]
    \centering
    \begin{tikzpicture}
        \node (solved) at (0,0) {};
        \node (sfti) at (1,-1) {};
        \node (la) at (3,1.5) {};
        \node (std) at (5,-1.8) {};
        \draw[fill=violet!70!orange!40, very thick] (3,0) ellipse (5.1 and 2.5);
        \node at (7,0) {$\S$};
        \draw[fill=violet!50!orange!40, dashed] (2,0) ellipse (3.8 and 2);
        \node  at (4.6,0) {$\Le[f_\solved]$};
        \draw[fill=violet!30!orange!40] (1,0) ellipse (2.5 and 1.5);
        \node at (2.5,0) {$\S^{\mathrm{face}}$};
        \fill[black] (solved) circle (1pt) node [left] {$f_\solved$};
        \fill[black] (sfti) circle (1pt);
        \fill[black] (la) circle (1pt);
        \fill[black] (std) circle (1pt);
        \path[draw,-stealth] (solved) to[bend right] node[midway, below left] {\scriptsize$\le\omega^2$} (sfti);
        \path[draw,-stealth] (solved) to[bend left] node[midway, below right] {\scriptsize$<\omega_1$} (la);
        \path[draw,-stealth] (std) to[bend left] node[midway, below] {\scriptsize$\le\omega^2$} (sfti);
    \end{tikzpicture}
    \caption{Space of all legal configurations of the countable edgeless cube $\Q_{\aleph_0}$ connected to $f_\solved$ by one- or two-way accessibility. The heavy border around the standard configurations is meant to suggest that there can be no other arrows in or out (except to illegal configurations). The dashed boundary between $\Le[f_\solved]$ and $\S$ is meant to suggest that this inclusion may or may not be proper.}
    \label{fig:config_space}
\end{figure}

\begin{theorem}
    For the countable edgeless cube $\Q_{\aleph_0}$ we have the following.
    \begin{enumerate}
        \item The inclusions
            \[
                \S^{\mathrm{face}} \subsetneq \Le[f_\solved] \subseteq \S
            \]
            hold where $\S$ is the set of standard configurations, $\S^{\mathrm{face}}$ is the set of standard configurations invariant under face twists, and $\Le[f_\solved]$ is the set of legal configurations accessible from $f_\solved$. 
        \item Every configuration in $\S^{\mathrm{face}}$ (including $f_\solved$) is accessible from every configuration in $\S$ in at most $\omega^2$ many moves. 
        \item No legal configuration outside $\S$ is connected by accessibility to $\S$ (in either direction).
    \end{enumerate}
    \label{thm:inclusions}
\end{theorem}

\begin{proof}
    The inclusion $\Le[f_\solved] \subseteq \S$ as well as property (3) are immediate conseqeunces of Observation~\ref{obs:four} and the fact that convergent sequences over any standard configuration preserve the relative positions of the colors of the center cluster. (This is an easy induction argument: a violation cannot first appear at a limit stage of a convergent sequence.)

    The claim that every $f \in \S^{\mathrm{face}}$ is accessible from every $f_0 \in \S$ in $\le\omega^2$ many moves is precisely the content of Lemma~\ref{lemma:std_access}. The special case $f_0 = f_\solved$ shows that $\S^{\mathrm{face}} \subseteq \Le[f_\solved]$. That this inclusion is proper is obvious, for example, applying any one slice twist to $f_\solved$ yields a configuration not invariant under face twists.
\end{proof}

\begin{remark}
    The above theorem is really two results masquerading as one. Namely, it applies to both the \emph{even} and \emph{odd} countable edgeless cubes, but standardness is a weaker condition for the even cube because there is no center cluster. In the even case, $\S$ includes all the legal configurations with exactly four of each color in each cluster, including, for instance, all the variations of the solved configuration with the the colors permuted. On the other hand, in the odd case, the legal configurations where each non-center cluster contains four cells of each color and the center cluster contains one of each (but in no partiuclar configuration) are partitioned into $6!/24 = 30$ components each analogous to Figure~\ref{fig:config_space} but, at its nucleus, a non-standard ``solved'' configuration with a different center cluster instead of $f_\solved$. Refinements of this picture may further distinguish the even and odd cases. For example, it is conceivable that the inclusion $\Le[f_\solved] \subseteq \S$ is proper in one case but not the other.
\end{remark}

Before concluding this section, we wish to comment on the possibility that the inclusion $\Le[f_{\solved}] \subseteq \S$ is proper. Firstly, the assertion that $\Le[f_\solved] = \S$ is simply the claim that every standard configuration is accessible from the solved configuration. We view the truth of this claim as a very natural question. Curiously though, in light of Theorem~\ref{thm:inclusions}, it is equivalently the assertion that every solvable configuration is accessible, an unexpected converse to our original motiving question: Is every accessible configuration solvable?

Suppose there is a configuration $f \in \S\setminus \Le[f_\solved]$, i.e., a standard configuration not accessible from the solved configuration. What can we say about $f$? It must have infinitely many clusters whose configurations are not invariant under some face twist. Otherwise, we could solve all the other clusters, then solve the finitely many aysymtrical ones at the end. However, this condition does suffice for inaccessiblilty. In fact, the diagonal clusters can be put into any standard configuration whatsoever. Let's see this.

Consider any diagonal cluster $C$. Mark the like-colored cells of the solved configuration so as to distinguish them. Then every marked configuration of $C$ corresponds to a unique permutation in $S_C \cong S_{24}$. If we don't worry about affecting other clusters, then we can check easilty taht every permutation of $C$ is possible by the usual slice and face twists. The slice moves alone generate a proper subgroup $\Sigma < S_C$ of the permuations. (This subgroup is proper only for the diagonal clusters, curiously.) Let $\Sigma \rho_1, \dots, \Sigma\rho_r$ enumerate the right cosets of $\Sigma$. Let $h_\alpha$ be any desired standard cluster configuration for each diagonal cluster $C(\alpha,\alpha)$. Say $h_\alpha$ is of type $j$ if $h_\alpha = \sigma\rho_jf_\solved$ for some $\sigma \in \Sigma$. For each $1 \le j \le r$, the permutation $\rho_j$ can be acheived by some cluster move sequence say with type sequences $a_1,\dots,a_k$, $b_1,\dots,b_k$, and $c_1,\dots,c_k$. 
We may prepare all the diagonal clusters we wish to put in a configuration of type $j$, say having indices $\alpha \in I_j$, as follows. Let
\[
    s_{j,i} = \seq{S_{a_i,\alpha} : \alpha \in I_j} \concat \seq{S_{b_i,\alpha} : \alpha \in I_j} \concat \seq{F_{c_i}}
\]
and
\[
    s_j = s_{j,1}\concat \cdots \concat s_{j,k}.
\]
Notice that this sequence is twist-finite and it affects the $\alpha$ diagonal cluster for each $\alpha \in I_j$ only by the subsequence $\seq{S_{a_1,\alpha}, S_{b_1,\alpha}, F_{c_1},\dots, S_{a_k,\alpha}, S_{b_k,\alpha}, F_{c_k}}$, which by construction applies the permutation $\rho_j$ to this cluster. After having done this, we can put each cluster $C(\alpha,\alpha)$ in the desired configuration $h_\alpha$ using no face twists and only slice moves of the $\pm\alpha$ layers since the necessary permutation is an element of $\Sigma$ (which is generated by the slice moves). Doing this for every $\alpha$ (in any order) is clearly twist-finite, hence convergent. 

Thus, for any desired standard configurations of the diagonal clusters, there is a legal accessible configuration acheiving exactly those cluster configurations. Again, the point is that it does not suffice for membership in $\Le[f_\solved]$ to be standard and face-twist invariant except possibly in finitely many clusters. The argument we jsut gave is not special to the diagonal clusters either, one can do the same thing more generally for any set of clusters whose $\D$ representtives intersect each row and column of $\D$ at most once.

Moreover, if there is any inaccessible standard $f$, it must have the following property. For every $g \in \S^{\mathrm{face}}$, if $\Delta \subseteq \D$ is the set of all $(x,y) \in \D$ such that $f$ and $g$ disagree on the cluster of $(x,y)$, then both projections $X,Y$ of $\Delta$, i.e., the minimal sets for which $\Delta \subseteq X \times Y$, must be infinite. This is a slightly stronger property than what we claimed above. Namely, the set of clusters in which $f$ is not face-twist invariant (identified with there represntatives in $\D$) is not merely infinite, but infinite in both coordinates. 

For those inspired to think further about the problem of whether $\Le[f_\solved] = \S$. We offer some candidate standard configurations which we think might be inaccessible, although we do not yet know how to show this. There is a well-known family of patterns for $\LLL N$ cubes called the \emph{superflip}. There are at least two distinct standard configurations of the odd cube $\Q_{\aleph_0}$ which are reasonable generalizations the superflip. The first, call it the $\omega$-superflip, can be described by ordering $L$ like $\omega$, as we did above. Then in each face, the $N\times N$ center square for each odd $N$ looks just like the coressponding face of the $\LLL N$ superflip. The other, call it the $\omega^*$-superflip, can be described by ordering $L$ like $\omega^*$, the reverse order of $\omega$. (Remember, this ordering is not intrinsic to the structure of $\Q_L$, it's merely an aide in describing the configuration. In particular, $\Q_L$ is still edgeless in the sense that each face twist only affects the cells of that face.) The diagonal clusters are as in the first superflip configurations but now the alternating patten propogates from the diagonals towards the cross, so to speak. Meanwhile, the cross also alternates colors (in the extrinsic $\omega^*$ order) with the outermost cell having the color of the center. These are genuinely distinct configurations, not a common configuration with different extrinsic orderings of $L$. One way to see this is the following. By our convention, the Front and Upper faces are white and blue, respectively. In both superflip configurations, the diagonals of the Front face are white and the upper triangle of the Front face is checkered with blue cells. In the $\omega$-superflip, each row of upper half of the Front face has only finitely many blue cells whereas in the $\omega^*$-superflip, each such row has only finitely many non-blue cells. It's easy to see that the standard algorithm for obtaining the $\LLL N$ superflip from the solved configuration generalized to $\Q_L$ using either extrinsic ordering (and accounting for the global rotations appropriately) simply does not converge over $f_\solved$. Of course, this is not a proof that there isn't some other way to obtain these configurations. A notable feature of these configurations is that an entire checkerboard of cell clusters in each face are not invariant under twists of that face and, similarly, in each slice, an alternating pattern of clusters are not invariant under twists of that slice.

\section{Open questions}
\label{sec:questions}
We conclude by collecting some of the natural questions prompted by this exploration.
\begin{enumerate}
    \item Is every convergent scramble of the uncountable edgeless cube $\Q_L$ for $\card L > \aleph_0$ solvable, in principle?
    \item In the countable edgeless case, $\Q_{\aleph_0}$, is the inclusion $\Le[f_\solved] \subseteq \S$ proper? That is, are there inaccessible standard configurations? And if not, can we simply characterize the accessible ones?
    \item Is there a simple characterization in the edgeless case of convergence over the solved configuration or of the accessible configurations and their algebraic structure, as there was for the edged case?
    \item In the edged case, since a basic sequence is convergent if and only if it is twist-finite, there is no need for half-turn twists, they can all be replaced by two quarter-turn twists without affecting convergence. Is the same true in the edgeless case or are there configurations accessible only using both half-turn and quarter-turn twists?
    \item Is there a uniform method to solve the edged cube that does not require knowledge of the scramble (suitably understood) or a sense in which that is impossible?
    \item Is the result of Corollary~\ref{cor:GN_edgeless} optimal? Clearly the optimal bound is at least $\omega$, but is there a method which works uniformly in $\omega \cdot n$ many moves or less for some $n$ for all accessible configurations? Perhaps every accessible configuration admits a solution in $\omega\cdot n$ for some $n$ but their supremum over all configurations in nonetheless $\omega^2$? Or is there an accessible configuration which requires $\omega^2$ many moves to solve?
\end{enumerate}

Note for several of these questions, what we have in mind is really a reduction in logical complexity. Generally, it seems unproblematic to quantify over cells or basic twists, but in many cases, the notions which arise naturally, such as the equivalence $\sim$ or the solvability/accessibility of a configuration, involve quantifying over configurations/labellings (functions on $\Q_L$) or sequences of twists. In these cases, we gain insight by eliminating these quantifiers. Indeed, some of the key of the results in this paper amount to such reductions. For example, the equivalence relation $\sim$ on the monoid $\m_L$ of sequences was defined in terms of a universal quantifier over configurations but Lemma~\ref{lemma:equiv_id} shows that, for universally convergent sequences, we can instead consider just the identity labelling, thus eliminating that quantifier. Lemma~\ref{lemma:std_access} shows that a configuration of $\Q_{\aleph_0}$ is solvable, a condition involving the existence of a basic sequence with a certain property, just in case it's standard, thus eliminating that existential quantifier.

\section{Bonus: coding well-orders into accessible configurations}
\label{sec:coding_orders}

The author considered the following curiosity as an initial stab at showing optimality of Corollary~\ref{cor:edged_GN}. Such a conjecture seems possibly misguided in light of Corollary~\ref{cor:GN_edgeless} but the following observation is neat anyway, so we include it here.

\begin{observation}
    Every well-ordering of $L$ can be coded into an accessible configuration of $\bar\Q_L$ or $\Q_L$. That is, there is an injection from the ordinals $[\card L, \card L^+)$ into the accessible configurations.
\end{observation}

\begin{proof}
    Let $\preceq$ be a well-order of $L$ of order type $\theta \in [\card L, \card L^+)$ and $\alpha_0,\alpha_1,\dots,\alpha_\xi,\dots$ for $\xi < \theta$ the corresponding enumeration. We code $\preceq$ into an accessible configuration of $\bar\Q_L$ as follows. Recall, in our usual coordinate conventions, the Front, Left, and Upper faces have colors $\white$, $\orange$, and $\blue$, respectively in the solved configuration. Now consider the twist-finite (hence universally convergent) basic sequence which applies $T_{x,+\alpha_\xi}$ then $T_{y,+\alpha_\xi}$ for each $\alpha_\xi$ in order of the enumeration. What effect has this sequence had on the Front face? An example of an intermediate configuration with three each of the affected rows and column visible is shown in Figure~\ref{fig:WO_encoding}. In the final configuration $f$, every cell in the lower left quadrant is left unchanged with color $\white$. Every cell in the upper left has color $\orange$ and, similarly, every cell in the lower right is $\blue$. The interesting part is the upper right. The diagonal is entirely $\orange$ (since each $y$ twist followed the matching $x$ twist), but the dual cells $(x,y)=(\alpha,\beta)$ and $(x,y)=(\beta,\alpha)$ are always one $\orange$ and the other $\blue$. The important point is that $\alpha \prec \beta$ if and only if $f(\alpha,\beta) = \orange$ (and $f(\beta,\gamma) = \blue$). To see that this is the case, consider the stage at which we are about to perform the $\alpha_\xi$ twists and have already performed the $\alpha_\eta$ twists for all $\eta < \xi$. At this point, both the $x = +\alpha_\xi$ layer of the Upper face and the $y = \alpha_\xi$ layer of the Left face are pristine, none of our earlier twists have affected these layers. So $T_{x,+\alpha_\xi}$ twists an entire column of $\blue$ cells into the $x = +\alpha_\xi$ layer of the Front face, in particular, turning $(x,y) = (\alpha_\xi,\alpha_\eta)$ blue for every earlier $\eta < \xi$. Then, the subsequent $T_{y,+\alpha_\xi}$ turns the whole $y = +\alpha_\xi$ row of the Front face orange, in particular, all the $(x,y) = (\alpha_\eta,\alpha_\xi)$. None of those particular cells are affected by any of the later twists in the sequence. 
    \begin{figure}[t]
        \centering
        \begin{tikzpicture}[scale=.2]
            \pgfmathsetmacro\r{9}
            \pgfmathsetmacro\a{3}
            \pgfmathsetmacro\b{8}
            \pgfmathsetmacro\g{6}
            \fill[blue] (\a,-\r) rectangle (\a+1,\r+1);
            \fill[orange] (-\r,\a) rectangle (\r+1,\a+1);
            \fill[blue] (\b,-\r) rectangle (\b+1,\r+1);
            \fill[orange] (-\r,\b) rectangle (\r+1,\b+1);
            \fill[blue] (\g,-\r) rectangle (\g+1,\r+1);
            \fill[orange] (-\r,\g) rectangle (\r+1,\g+1);
            \draw (-\r,-\r) grid (\r+1,\r+1);
            \node [below] at (.5,-\r) {$\mathstrut 0$};
            \node [below] at (\a+.5,-\r) {$\mathstrut \alpha$};
            \node [left] at (-\r,\a+.5) {$\mathstrut \alpha$};
            \node [below] at (\b+.5,-\r) {$\mathstrut \beta$};
            \node [left] at (-\r,\b+.5) {$\mathstrut \beta$};
            \node [below] at (\g+.5,-\r) {$\mathstrut \gamma$};
            \node [left] at (-\r,\g+.5) {$\mathstrut \gamma$};
            \node [left] at (-\r,.5) {$0$};
            \foreach \x/\y in {3/8,8/6,6/3} {
                \draw [thick, green] (\x,\y) rectangle (\x+1,\y+1);
                \draw [thick, green] (\y,\x) rectangle (\y+1,\x+1);
            }
            \draw [<->,thick] (.5,\r+2) node[above] {$+y$} -- (.5,.5) -- (\r+2,.5) node[right] {$+x$};
            \draw [dashed] (-\r,-\r) -- (\r+1,\r+1);
        \end{tikzpicture}%
        \begin{tikzpicture}[scale=.2]
            \pgfmathsetmacro\r{9}
            \foreach \a in {3,8,6,1,2,7,5,9,4} {
                \fill[blue] (\a,-\r) rectangle (\a+1,\r+1);
                \fill[orange] (-\r,\a) rectangle (\r+1,\a+1);
            }
            \draw (-\r,-\r) grid (\r+1,\r+1);
            \node [below] at (.5,-\r) {$\mathstrut 0$};
            \node [left] at (-\r,.5) {$0$};
            \draw [<->,thick] (.5,\r+2) node[above] {$+y$} -- (.5,.5) -- (\r+2,.5) node[right] {$+x$};
            \foreach \x/\y in {3/8,8/6,6/3} {
                \draw [thick, green] (\x,\y) rectangle (\x+1,\y+1);
                \draw [thick, green] (\y,\x) rectangle (\y+1,\x+1);
            }
            \draw [dashed] (-\r,-\r) -- (\r+1,\r+1);
        \end{tikzpicture}%
        \caption{Configuration encoding a well-order $\preceq$ of $L$. On the left, the configuration after only $T_{x,\alpha},T_{y,\alpha},T_{x,\beta},T_{y,\beta},T_{x,\gamma},T_{y,\gamma}$ (because $\alpha \prec \beta \prec \gamma$) with the cross cells highlighted. On the right, (a piece of) the final configuration with the same cells highlighted to illustrate that they are unaffected by later twists in the sequence.}
        \label{fig:WO_encoding}
    \end{figure}
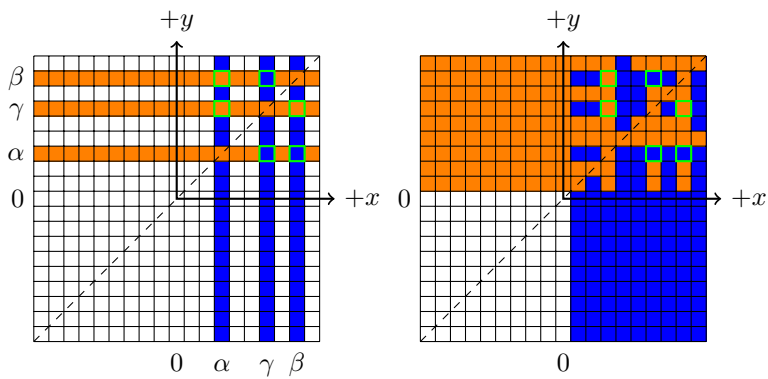
    Thus, this configuration codes the well-order $\preceq$. 
\end{proof}

\bibliography{references}
\bibliographystyle{plain}

\end{document}